\documentclass[11pt]{article}

\usepackage[margin=1in]{geometry}
\usepackage{amsmath,amsthm,amssymb}
\usepackage{enumerate}
\usepackage{xfrac}
\usepackage[colorlinks=true,hidelinks]{hyperref}
    \AtBeginDocument{}
\usepackage{tikz} 
\usetikzlibrary{arrows}
\usetikzlibrary{patterns}
\usepackage{pgfplots} 
\usetikzlibrary{decorations.markings} 
\usetikzlibrary{3d,calc} 
\usepackage{pifont}
\DeclareFontFamily{OT1}{pzc}{}
  \DeclareFontShape{OT1}{pzc}{m}{it}{<-> s * [1.200] pzcmi7t}{}
  \DeclareMathAlphabet{\mathpzc}{OT1}{pzc}{m}{it}
\usepackage{fancyhdr}
\newcommand{\Z}{\mathbb{Z}}
\newcommand{\Q}{\mathbb{Q}}

\newcommand{\C}{\mathbb{C}}
\newcommand{\F}{\mathbb{F}}
\renewcommand{\L}{\mathcal{L}}

\renewcommand{\P}{\mathbb{P}}
\newcommand{\M}{\mathcal{M}}

\newcommand{\G}{\mathbb{G}}
\newcommand{\X}{\mathcal{X}}
\newcommand{\Y}{\mathcal{Y}}
\newcommand{\Zz}{\mathcal{Z}}
\DeclareFontFamily{U}{wncy}{}
    \DeclareFontShape{U}{wncy}{m}{n}{<->wncyr10}{}
    \DeclareSymbolFont{mcy}{U}{wncy}{m}{n}
    \DeclareMathSymbol{\Sha}{\mathord}{mcy}{"58} 

\renewcommand{\char}{\ensuremath{\operatorname{char}}}
\newcommand{\Spec}{\ensuremath{\operatorname{Spec}}}

\newcommand{\Aut}{\ensuremath{\operatorname{Aut}}}

\newcommand{\la}{\langle}
\newcommand{\ra}{\rangle}
\newcommand{\orb}{\mathcal{O}}

\newcommand*\cat[1]{{\tt #1}}

\newcommand{\Hom}{\ensuremath{\operatorname{Hom}}}

\newcommand{\Gal}{\ensuremath{\operatorname{Gal}}}

\newcommand{\Sym}{\ensuremath{\operatorname{Sym}}}

\newcommand{\invlim}{\displaystyle\lim_{\longleftarrow}}

\newcommand{\A}{\mathbb{A}}

\renewcommand{\div}{\ensuremath{\operatorname{div}}}

\newcommand{\et}{\ensuremath{\operatorname{\acute{e}t}}}
\newcommand{\Proj}{\ensuremath{\operatorname{Proj}}}
\newcommand{\W}{\mathbb{W}}

\newcommand{\DIV}{\frak{Div}}

\newcommand{\rig}{\ensuremath{\operatorname{rig}}}

\newcommand{\U}{\mathcal{U}}

\renewcommand{\epsilon}{\varepsilon}

\newtheorem{thm}{Theorem}[section]
\newtheorem{prop}[thm]{Proposition}
\newtheorem{lem}[thm]{Lemma}
\newtheorem{defn}[thm]{Definition}
\newtheorem{cor}[thm]{Corollary}

\newtheorem{rem}[thm]{Remark}
  \let\oldrem\rem
  \renewcommand{\rem}{\oldrem\normalfont}
\newtheorem{question}[thm]{Question}
\newtheorem{ex}[thm]{Example}
  \let\oldex\ex
  \renewcommand{\ex}{\oldex\normalfont}

\makeatletter
\let\@fnsymbol\@arabic
\makeatother

\begin{document}




\title{Artin--Schreier--Witt Theory for Stacky Curves}
\author{Andrew Kobin\thanks{The author is partially supported by the American Mathematical Society and the Simons Foundation.}}

\maketitle


\begin{abstract}

We extend our previous classification of stacky curves in positive characteristic using higher ramification data and Artin--Schreier--Witt theory. The main new technical tool introduced is the Artin--Schreier--Witt root stack, a generalization of root stacks to the wildly ramified setting. We then apply our wild Riemann--Hurwitz theorem for stacks to compute the canonical rings of some wild stacky curves. 

\end{abstract}


\section{Introduction}

\vspace{0.1in}
Classical algebraic geometry in characteristic $p > 0$ already presents a wealth of new phenomena that do not arise in characteristic $0$. Consider for instance the topology of the complex plane, viewed as the affine curve $\A_{\C}^{1}$. Since $\A_{\C}^{1}$ is simply connected, it has no nontrivial coverings; it is not until one removes points from $\A_{\C}^{1}$ that more interesting topology begins to appear. In contrast, for an algebraically closed field $k$ of characteristic $p > 0$, the affine line $\A_{k}^{1}$ is far from being simply connected: Abhyankar's conjecture (a theorem of Harbater \cite{harb-abhyankar} and Raynaud \cite{ray}) describes the finite quotients of the \'{e}tale fundamental group $\pi_{1}(\A_{k}^{1})$, but this profinite group is not even prosolvable. 

The key observation in studying these phenomena is that \'{e}tale covers of $\A_{k}^{1}$ correspond to covers of $\P_{k}^{1}$ which are ramified over the point at infinity. In characteristic $p > 0$, ramified covers of curves (or equivalently, function field extensions) can be studied using various {\it ramification filtrations} of their Galois groups. For example, by Artin--Schreier theory, $\Z/p\Z$-extensions of a perfect field $K$ of characteristic $p$ are all of the form 
$$
L = K[x]/(x^{p} - x - a) \quad\text{for some } a\in K,a\not = b^{p} - b \text{ for any } b\in K. 
$$
If $K$ is a discretely valued field with valuation $v$, the integer $m = -v(a)$ coincides with the {\it jump} in the ramification filtration of $\Gal(L/K)$. This jump is an isomorphism invariant of the extension and (after completion) essentially classifies degree $p$ extensions. This situation can be understood geometrically as follows. When $K$ is a function field corresponding to a curve $C$, then $\Z/p\Z$-extensions $L/K$ are equivalent to $\Z/p\Z$-covers $D\rightarrow C$, up to birational equivalence, and each of these covers can be obtained by pulling back the Artin--Schreier isogeny $\wp : \G_{a}\rightarrow\G_{a},x\mapsto x^{p} - x$ along a map $C\rightarrow\G_{a}$. 

A geometric description of the ramification jump $m$ requires more work. Assume $D\rightarrow C$ has a single branch point $P\in C$ and consider instead a map $h : C\rightarrow\P^{1}$, where $\P^{1}$ is viewed as the one-point compactification of $\G_{a}$ and $P$ maps to the distinguished point $\infty$ in $\P^{1}$. The Artin--Schreier isogeny on $\G_{a}$ extends to a degree $p$ map $\Psi_{1} : \P^{1}\rightarrow\P^{1}$ and one shows that the cover $D\rightarrow C$ may be obtained by pulling back $\Psi_{1}$ along $h$. It also follows that $m$ is precisely the order of vanishing of $h$ at $P$. 

Artin--Schreier--Witt theory generalizes Artin--Schreier theory to the case of $\Z/p^{n}\Z$-extensions of $K$ for $n\geq 2$. Namely, these are all of the form 
$$
L = K[\underline{x}]/(\underline{x}^{p} - \underline{x} - \underline{a}) \quad\text{for some } \underline{a}\in\W_{n}(K),\underline{a}\not = \underline{b}^{p} - \underline{b} \text{ for any } \underline{b}\in\W_{n}(K). 
$$
Here, $\W_{n}(K)$ is the ring of length $n$ $p$-typical Witt vectors over $K$ and $\underline{x} = (x_{0},x_{1},\ldots,x_{n - 1})$ is a Witt vector of indeterminates. When $K$ is a function field, extensions $L/K$ are obtained by pulling back the Artin--Schreier--Witt isogeny 
$$
\wp : \W_{n} \longrightarrow \W_{n}, \quad \underline{x} \longmapsto \underline{x}^{p} - \underline{x}
$$
along a map $C\rightarrow\W_{n}$, where $C$ is a curve with function field $K$. 

To study the ramification invariants geometrically, Garuti \cite{gar} introduced a compactification $\overline{\W}_{n}$ of the ring $\W_{n}$ which plays the same role for cyclic $p^{n}$-covers as $\P^{1}$ played for $p$-covers in the above paragraph. Concretely, $\wp$ extends to a degree $p^{n}$ map $\Psi_{n} : \overline{\W}_{n}\rightarrow\overline{\W}_{n}$ and $\Z/p^{n}\Z$-covers of curves $D\rightarrow C$ can be obtained by pulling back $\Psi_{n}$ along a map $h : C\rightarrow\overline{\W}_{n}$. Then, the $n$ different jumps in the ramification filtration of $\Gal(D/C)\cong\Z/p^{n}\Z$ coincide with the orders of vanishing of $h$ along the pullbacks of various divisors in $\overline{\W}_{n}$ \cite[Thm.~1]{gar}.

\subsection{Stacks in Characteristic $p$}
\label{sec:introstackscharp}

In \cite{kob}, the author introduced a construction called an {\it Artin--Schreier root stack} in order to study $\Z/p\Z$-covers of curves using stacks and to classify stacky curves with wild ramification of order $p$. Briefly, if $D\rightarrow C$ is a cover of curves branched at $P\in C$ such that the inertia group at $P$ is $I\cong \Z/p\Z$ (as algebraic groups), let $m$ be the ramification jump of the ramification filtration of $I$. Then \'{e}tale-locally, the corresponding map $h : C\rightarrow\P^{1}$ taking $P$ to $\infty$ factors through the weighted projective line $\P(1,m)$, which admits a degree $p$ map $\P(1,m)\rightarrow\P(1,m)$. This map descends to the quotient stack, $\wp_{m} : [\P(1,m)/\G_{a}]\rightarrow [\P(1,m)/\G_{a}]$, and pulling back $\wp_{m}$ along $h$ defines the {\it Artin--Schreier root stack} of $C$, denoted $\wp_{m}^{-1}((L,s,f)/C)$. This definition is made global in \cite[Def.~6.9]{kob}. 

One of the main applications of this construction, \cite[Thm.~6.16]{kob}, shows that every such cover of curves $D\rightarrow C$ factors \'{e}tale-locally through an Artin--Schreier root stack which is a wild stacky curve. Another, \cite[Thm.~6.18]{kob}, classifies wild stacky curves with this type of inertia. 

When the cover of curves (or instead, the wild stacky curve) has inertia of order $p^{n}$ for some $n\geq 2$, it is always possible to iterate the Artin--Schreier root stack construction to obtain the desired stacky structure \cite[Lem.~6.11]{kob}. However, the local equations/geometric data quickly becomes messy (as with ordinary curves). In the cyclic case, we would like to directly generalize the construction in \cite{kob}, rather than having to take towers of Artin--Schreier roots. This leads us to Garuti's geometric version of Artin--Schreier--Witt theory described in the introduction. 

In Section~\ref{sec:AWSrootstacks}, we introduce a stacky version $\overline{\W}_{n}(\overline{m})$ of Garuti's compactification which then allows us to define the {\it Artin--Schreier--Witt root stack} of a scheme $X$ along a map $X\rightarrow [\W_{n}(\overline{m})/\W_{n}]$. Here, $\overline{m} = (m_{1},\ldots,m_{n})$ is a sequence of positive integers related to the ramification jumps of the ramified covers of $X$ one wants to allow through this stacky structure. As a functor, $\overline{\W}_{n}(\overline{m})$ generalizes the $n = 1$ case $\overline{\W}_{1}(m) = \P(1,m)$, the weighted projective stack whose functor of points is described from this perspective in \cite[Prop.~6.4]{kob}. For $n\geq 2$, a map $X\rightarrow\overline{\W}_{n}(\overline{m})$ is determined by a tuple $(L,s,f_{1},\ldots,f_{n})$, where $L$ is a line bundle on $X$, $s$ is a section of $L$ and $f_{i}$ is a section of $L^{\otimes m_{i}}$; see Proposition~\ref{prop:ASWlinebundlesectionclassification}. The resulting root stack is denoted $\Psi_{\overline{m}}^{-1}((L,s,f_{1},\ldots,f_{n})/X)$. 

The simple reason for keeping track of all this extra data is that wildly ramified structures (covers of curves, stacks, etc.) are more diverse than tame structures and require more invariants to classify. This is already evident in the $n = 1$ case \cite[Rem.~6.19]{kob} and will play a role in the classification results of the present article, summarized in the following two theorems. 

\begin{thm}[{Theorem~\ref{thm:factorthroughASW}}]
Suppose $Y\rightarrow X$ is a finite separable Galois cover of curves over an algebraically closed field of characteristic $p > 0$, with a ramification point $y\in Y$ over $x\in X$ having inertia group $I(y\mid x)\cong\Z/p^{n}\Z$. Then \'{e}tale-locally, $\varphi$ factors through an Artin--Schreier--Witt root stack $\Psi_{\overline{m}}^{-1}((L,s,f_{1},\ldots,f_{n})/X)$. 
\end{thm}

\begin{thm}[{Theorem~\ref{thm:wildstacky}}]
\label{thm:wildstackyintro}
Let $\X$ be a stacky curve over a perfect field of characteristic $p > 0$. Then for any stacky point $x$ with cyclic automorphism group of order $p^{n}$, there is an open substack $\mathcal{Z}\subseteq\X$ containing $x$ which is isomorphic to $\Psi_{\overline{m}}^{-1}((L,s,f_{1},\ldots,f_{n})/Z)$ for some $\overline{m}$ and $L,s,f_{1},\ldots,f_{n}$ on an open subscheme $Z$ of the coarse space of $\X$. 
\end{thm}

In Section~\ref{sec:universalstack}, we also package together the collection of $\Psi_{\overline{m}}^{-1}((L,s,f_{1},\ldots,f_{n})/X)$ into a universal Artin--Schreier--Witt root stack $\mathcal{ASW}_{X}$ and give a unified description of $\Z/p^{n}\Z$-covers in Theorem~\ref{thm:univASW}. 

Theorem~\ref{thm:wildstackyintro} generalizes to the case of a stacky point with cyclic-by-tame automorphism group (see Theorem~\ref{thm:cyclicbytame}), but more work will be needed to handle stacky points with more general automorphism groups. By classical ramification theory \cite[Ch.~IV]{ser2}, these can be of the form $P\rtimes\Z/r\Z$ where $P$ is a $p$-group and $r$ is prime to $p$. By iterating tame and wild root stacks, one can achieve many desired stacky structures. It is unclear how to globalize this procedure, as we do with each individual root stack using $[\A^{1}/\G_{m}]$ and $[\overline{\W}_{n}(\overline{m})/\W_{n}]$. However, see Section~\ref{sec:future} for a possible approach.

\subsection{Application: Canonical Rings of Stacky Curves}

In classical algebraic geometry, the {\it canonical ring} of a projective curve $X$ is defined as the graded ring 
$$
R(X) = \bigoplus_{k = 0}^{\infty} H^{0}(X,\omega_{X}^{\otimes k}),
$$
where $\omega_{X}$ is the canonical sheaf. The canonical ring contains important information about the geometry of $X$; for example, when $X$ is smooth of genus at least $2$, $\Proj R(X)$ is a model for $X$. Explicit descriptions of $R(X)$ exist, such as Petri's theorem (cf.~\cite[Sec.~1.1]{vzb}), which in turn provide explicit equations for $X$ inside projective space. 

Replacing $X$ with a stacky curve $\X$, one can similarly define a canonical ring $R(\X)$ in order to study models of $\X$ inside weighted projective space. Generalizing results like Petri's theorem, Voight and Zureick-Brown provide generators and relations for $R(\X)$ when $\X$ is a tame log stacky curve \cite[Thm.~1.4.1]{vzb}. 

For number theorists, one of the most useful applications of theorems like {\it loc.~cit.}~is to modular forms. When $\X$ is a modular stacky curve (that is, a modular curve with stacky structure encoding the automorphisms of elliptic curves with a given level stucture), a logarithmic version of $R(\X)$ is isomorphic to a graded ring of modular forms and the description in {\it loc.~cit.}~recovers formulas for generators and relations of rings of modular forms. Notably, this description holds in all characteristics, as long as the modular curve has no wild ramification. Nevertheless, many modular curves have wild ramification in characteristic $p$, such as $X(1)$ in characteristics $2$ and $3$, and therefore the results of \cite{vzb} do not apply. 

In \cite{kob}, we began investigating canonical rings of wild stacky curves. The starting place is a stacky Riemann--Hurwitz formula that holds in all characteristics: 

\begin{thm}[Stacky Riemann--Hurwitz, {\cite[Prop.~7.1]{kob}}]
\label{thm:RH}
For a stacky curve $\X$ with coarse moduli space $\pi : \X\rightarrow X$, the canonical divisors $K_{\X}$ and $K_{X}$ are related by the formula 
$$
K_{\X} = \pi^{*}K_{X} + \sum_{x\in\X(k)}\sum_{i = 0}^{\infty} (|G_{x,i}| - 1)x. 
$$
Here, $G_{x,i}$ is the $i$th group in the higher ramification filtration of the automorphism group $G_{x}$ at $x$. 
\end{thm}

Since the canonical sheaf $\omega_{\X}$ is the line bundle attached to the divisor $K_{\X}$, this result is one of the main tools for computing the canonical ring of $\X$ in any characteristic. An explicit example of $K_{\X}$ for a wild stacky curve is computed in \cite[Ex.~7.8]{kob}. At the time, the structure theory of wild stacky curves (in particular, their local root stack structure) was only developed for stacks with wild automorphism groups isomorphic to $\Z/p\Z$. The main results in this article allow us to extend the approaches in \cite{vzb,kob} to more general stacky curves. 

In particular, one would like descriptions of rings of modular forms like those of \cite[Ch.~6]{vzb} when the relevant modular curve is a wild stacky curve. In characteristics $2$ and $3$, the modular curve $\X(1)$ is wildly ramified at $j = 0 = 1728$ (see Examples~\ref{ex:M11char3} and~\ref{ex:M11char2}) and this produces wild ramification in many other modular curves, such as $\X_{0}(N)$ for certain $N$. Another example noted in \cite[Rmk.~5.3.11]{vzb} is the quotient $[X(p)/PSL_{2}(\F_{p})]$ in characteristic $3$, which is a stacky $\P^{1}$ with two stacky points, one having tame automorphism group $\Z/p\Z$ and the other having wild automorphism group $S_{3}$. We will compute canonical divisors for these curves in Section~\ref{sec:canring}. In a forthcoming article with David Zureick-Brown, we will give a description of the corresponding rings of modular forms for many stacky modular curves using this theory.

\subsection{Relation to Other Work}

Moduli spaces of wildly ramified curves in characteristic $p > 0$ have been studied in a number of places. In \cite{pri}, the author constructs a moduli space for $G$-covers of curves with inertia groups of the form $\Z/p\Z\rtimes\Z/r\Z$ and prescribed ramification jumps. In particular, [{\it loc.~cit.}, Thm.~3.3.4] describes the moduli of covers of $\P^{1}$ branched at one point, which is the situation we will analyze in detail in Example~\ref{ex:key} for inertia groups $\Z/p^{n}\Z$. To turn this moduli problem into a moduli \emph{stack}, one could replace the configuration space $(\G_{m}\times\G_{a}^{r - 1})/\mu_{p - 1}$ from [{\it loc.~cit.}, Def.~2.2.5] with the quotient stack $[(\G_{m}\times\G_{a}^{r - 1})/\mu_{p - 1}]$, whose coarse space is the configuration space. It is likely that certain substacks of this stack correspond to refinements of the moduli problem of $G$-covers. 

Along these lines, the authors in \cite{dh} stratify the moduli space of $\Z/p^{n}\Z$-covers by specifying the sequence of conductors in the tower of $\Z/p\Z$-subcovers. These strata are refined moduli problems represented by algebraic stacks [{\it loc.~cit}, Prop.~3.4, Cor.~3.5] and the authors identify irreducible components of these stacks. It is likely that their moduli stacks have connections to the stacks described in Section~\ref{sec:universalstack}, though we will leave such a description to future work. 

The stacks in Section~\ref{sec:universalstack} also has connections to the moduli stacks of formal $G$-torsors considered in \cite{ty}. In particular, when $G = \Z/p^{n}\Z$, these moduli stacks can be filled out by Artin--Schreier--Witt stacks; see Example~\ref{ex:TY}. 

Finally, the structure theorem~\ref{thm:wildstackyintro} can be viewed as a wild analogue of the structure theory in \cite{gs}, for stacky curves. Further work is needed to extend the theory beyond dimension $1$ and, as mentioned in Section~\ref{sec:future}, beyond the cyclic wild case.

\subsection{Outline of the Paper}

The paper is organized as follows. In Section~\ref{sec:stackycurves}, we recall the basic geometry of stacky curves. Section~\ref{sec:ASWtheory} is a brief survey of wild ramification and Artin--Schreier--Witt theory. To carry these tools over to stacky curves, we use a construction of Garuti \cite{gar} which is described in Section~\ref{sec:garuti}. The construction of Artin--Schreier--Witt root stacks is carried out in Section~\ref{sec:AWSrootstacks}, followed by our main classification theorems for wild stacky curves in Section~\ref{sec:mainthms}. Section~\ref{sec:universalstack} describes how to capture all Artin--Schreier--Witt covers of curves using a limit of Artin--Schreier--Witt root stacks. Finally, in Section~\ref{sec:canring}, we apply the results here and in \cite{kob} to compute several examples of canonical rings of stacky curves. 

The author would like to thank Andrew Obus and David Zureick-Brown for their guidance on this project. Particular thanks go to David for suggesting the proof of Lemma~\ref{lem:doublequotient}.


\section{Stacky Curves}
\label{sec:stackycurves}

In this section, we collect the basic definitions and properties for stacky curves needed for later sections. 

\subsection{Review of Stacks}

%
%

Let $\X$ be a Deligne--Mumford stack over a scheme $S$, i.e.~an algebraic stack admitting an \'{e}tale presentation $U\rightarrow\X$ where $U$ is a smooth $S$-scheme. The set of points of $\X$, denoted $|\X|$, is defined to be the set of equivalence classes of morphisms $x : \Spec k\rightarrow\X$, where $k$ is a field, and where two points $x : \Spec k\rightarrow\X$ and $x' : \Spec k'\rightarrow\X$ are said to be equivalent if there exists a field $L\supseteq k,k'$ such that the diagram 
\begin{center}
\begin{tikzpicture}[scale=1.5]
  \node at (0,0) (a) {$\Spec L$};
  \node at (1,1) (b1) {$\Spec k$};
  \node at (1,-1) (b2) {$\Spec k'$};
  \node at (2,0) (c) {$\X$};
  \draw[->] (a) -- (b1);
  \draw[->] (a) -- (b2);
  \draw[->] (b1) -- (c) node[above,pos=.6] {$x$};
  \draw[->] (b2) -- (c) node[above,pos=.5] {$x'$};
\end{tikzpicture}
\end{center}
commutes. The {\it automorphism group} of a point $x\in |\X|$ is defined to be the pullback $G_{x}$ in the following diagram: 
\begin{center}
\begin{tikzpicture}[xscale=2.2,yscale=2]
  \node at (0,1) (a) {$G_{x}$};
  \node at (1,1) (b) {$\Spec k$};
  \node at (0,0) (c) {$\X$};
  \node at (1,0) (d) {$\X\times_{S}\X$};
  \draw[->] (a) -- (b);
  \draw[->] (a) -- (c);
  \draw[->] (b) -- (d) node[right,pos=.5] {$(x,x)$};
  \draw[->] (c) -- (d) node[above,pos=.5] {$\Delta_{\X}$};
\end{tikzpicture}
\end{center}
A geometric point is a point $\bar{x} : \Spec k\rightarrow\X$ where $k$ is algebraically closed. 

\begin{rem}
\label{rem:DMfiniteaut}
Colloquially, a Deligne--Mumford stack is said to have finite automorphism groups. The technical fact is that an algebraic stack over $S$ with finitely presented diagonal is Deligne--Mumford if and only for every geometric point $\bar{x}$ of $\X$, the automorphism group $G_{\bar{x}}$ is a reduced, finite group scheme \cite[Thm.~8.3.3, Rmk.~8.3.4]{ols}. When $S = \Spec\bar{k}$ for an algebraically closed field $\bar{k}$, this is equivalent to saying each automorphism group $G_{\bar{x}}$ is finite. 
\end{rem}

A {\it stacky curve} is a smooth, separated, connected, one-dimensional Deligne--Mumford stack which is generically a scheme, i.e.~there exists an open subscheme $U$ of the coarse moduli space $X$ of $\X$ such that the induced map $\X\times_{X}U\rightarrow U$ is an isomorphism. 

Finally, when $S = \Spec k$, a {\it tame stack} is a stack $\X$ for which the orders of the (finite, by Remark~\ref{rem:DMfiniteaut}) automorphism groups of its points are coprime to $\char k$; otherwise, $\X$ is said to be a {\it wild stack}. 

\subsection{Quotients}
\label{sec:quotientstack}

Let $X$ be a smooth, projective $k$-scheme, where $k$ is a field. For a smooth group scheme $G\subseteq\Aut(X)$, the {\it quotient stack} $[X/G]$ is defined to be the category fibred in groupoids over $\cat{Sch}_{k}$ whose objects are triples $(T,P,\pi)$, where $T\in\cat{Sch}_{k}$, $P$ is a $G\times_{k}T$-torsor for the \'{e}tale site $T_{\et}$ and $\pi : P\rightarrow X\times_{k}T$ is a $G\times_{k}T$-equivariant morphism. Morphisms $(T',P',\pi')\rightarrow (T,P,\pi)$ in $[X/G]$ are given by compatible morphisms of $k$-schemes $\varphi : T'\rightarrow T$ and $G\times_{k}T'$-torsors $\psi : P'\rightarrow \varphi^{*}P$ such that $\varphi^{*}\pi\circ\psi = \pi'$. This is often summarized by the diagram 
\begin{center}
\begin{tikzpicture}[xscale=2.2,yscale=2]
  \node at (0,1) (a) {$P$};
  \node at (1,1) (b) {$X$};
  \node at (0,0) (c) {$T$};
  \node at (1,0) (d) {$[X/G]$};
  \draw[->] (a) -- (b) node[above,pos=.5] {$\pi$};
  \draw[->] (a) -- (c);
  \draw[->,dashed] (b) -- (d);
  \draw[->,dashed] (c) -- (d);
\end{tikzpicture}.
\end{center}

%
%
%

By \cite[11.3.1]{ols}, every stacky curve $\X$ is, \'{e}tale locally, a quotient stack $[U/G]$, where $G$ may be taken to be the automorphism group of a geometric point of $\X$, hence a finite group. It follows from ramification theory \cite[Ch.~IV]{ser2} (see also Section~\ref{sec:ram}) that when $\X$ is tame, every automorphism group of $\X$ is cyclic. As a result, tame stacky curves can be completely described by their coarse space, together with a finite list of {\it stacky points} (points with nontrivial automorphism groups) and the orders of their automorphism groups. 

In contrast, if $\X$ is wild, it may have noncyclic -- even nonabelian! -- automorphism groups, coming from higher ramification data (again, see Section~\ref{sec:ram}). The main goal of this article is to describe how wild stacky curves can still be classified by specifying data on their coarse space. 

The following result will be used later to construct isomorphisms between stacks. 

\begin{lem}
\label{lem:CFGequiv}
If $F : \X\rightarrow\Y$ is a functor between categories fibred in groupoids over $\cat{Sch}_{S}$, then $F$ is an equivalence of categories fibred in groupoids if and only if for each $S$-scheme $T$, the functor $F_{T} : \X(T)\rightarrow\Y(T)$ is an equivalence of categories. 
\end{lem}

\begin{proof}
This is a special case of~\cite[Tag 003Z]{sp}. 
\end{proof}

In Section~\ref{sec:ASWroots}, we will study towers of quotient stacks, for which we will make use of the following result. 

\begin{lem}
\label{lem:doublequotient}
Let $G$ be a group scheme acting on a scheme $X$ as in Subsection~\ref{sec:quotientstack} and let $H\subseteq G$ be a normal subgroup scheme. Then $[X/G] \cong [[X/H]/(G/H)]$. 
\end{lem}

\begin{proof}
By Lemma~\ref{lem:CFGequiv}, it is enough to check the isomorphism on groupoids $[X/G](T) \cong [[X/H]/(G/H)](T)$. At this level, the isomorphism is the identity on torsors $P$ and identifies the morphisms $P\rightarrow X\times_{k}T$ and $P\rightarrow [X/H]\times_{k}T$ via any fixed isomorphism between $G(T)$ and $H(T)\times(G/H)(T)$. 
\end{proof}



\subsection{Normalization for Stacks}

In this section, we recall the notions of normalization and relative normalization for stacks, following \cite[Sec.~3]{kob}; see also \cite[Appendix A]{ab}. 

\begin{defn}
Let $\X$ be a locally noetherian algebraic stack over $S$. Then $\X$ is {\bf normal} if there is a smooth presentation $U\rightarrow\X$ where $U$ is a normal scheme. The {\bf relative normalization} of $\X$ is an algebraic stack $\X^{\nu}$ and a representable morphism of stacks $\X^{\nu}\rightarrow\X$ such that for any smooth morphism $U\rightarrow\X$ where $U$ is a scheme, $U\times_{\X}\X^{\nu}$ is the relative normalization of $U\rightarrow S$. 
\end{defn}

\begin{lem}[{\cite[Lem.~A.5]{ab}}]
For a locally noetherian algebraic stack $\X$, the relative normalization $\X^{\nu}$ is uniquely determined by the following two properties: 
\begin{enumerate}[\quad (1)]
  \item $\X^{\nu}\rightarrow\X$ is an integral surjection which induces a bijection on irreducible components. 
  \item $\X^{\nu}\rightarrow\X$ is terminal among morphisms of algebraic stacks $\Zz\rightarrow\X$, where $\Zz$ is normal, which are dominant on irreducible components. 
\end{enumerate}
\end{lem}

\begin{defn}
\label{defn:normalpullback}
Let $\X,\Y$ and $\Zz$ be algebraic stacks and suppose there are morphisms $\Y\rightarrow\X$ and $\Zz\rightarrow\X$. Define the {\bf normalized pullback} $\Y\times_{\X}^{\nu}\Zz$ to be the relative normalization of the fibre product $\Y\times_{\X}\Zz$. 
\end{defn}

As in \cite{kob}, we will write the normalized pullback as a diagram 
\begin{center}
\begin{tikzpicture}[scale=2]
  \node at (0,1) (a) {$\Y\times_{\X}^{\nu}\Zz$};
  \node at (1,1) (b) {$\Zz$};
  \node at (0,0) (c) {$\Y$};
  \node at (1,0) (d) {$\X$};
  \draw[->] (a) -- (b);
  \draw[->] (a) -- (c);
  \draw[->] (b) -- (d);
  \draw[->] (c) -- (d);
  \node at (.3,.7) {$\nu$};
  \draw (.2,.6) -- (.4,.6) -- (.4,.8);
\end{tikzpicture}
\end{center}


\section{Artin--Schreier--Witt Theory and Cyclic Covers}
\label{sec:ASWtheory}

In \cite{kob}, the author's construction of the {\it Artin--Schreier root stack} solves the problem of taking $p$th roots of line bundles on a stacky curve in characteristic $p > 0$, but one may want to compute roots of a line bundle of arbitrary order (and we will see there is good motivation for this). As for local fields, a geometric version of Artin--Schreier--Witt theory will allow us to take $p^{n}$th roots of line bundles for $n > 1$. We give the basic outline of the theory in this section.

\subsection{Artin--Schreier Theory}

Suppose $k$ is a local field of characteristic $p > 0$ and $L/k$ is a Galois extension with group $G = \Z/p^{n}\Z$. When $n = 1$, such extensions are all of the form $L = k[x]/(x^{p} - x - a)$ for some $a\in k$, with isomorphism classes of extensions corresponding to the valuation $v(a)$. When $n = 2$, write $L/k$ as a tower $L\supseteq M\supseteq k$, where $L/M$ and $M/k$ are both Galois extensions with group $\Z/p\Z$. Then by Artin--Schreier theory, 
$$
M = k[x]/(x^{p} - x - a) \qquad\text{and}\qquad L = M[z]/(z^{p} - z - b)
$$
for $a\in k\smallsetminus\wp(k)$ and $b\in M\smallsetminus\wp(M)$ -- here, $\wp$ denotes the map $c\mapsto c^{p} - c$. It turns out (see \cite{op}) that the extension $L/k$ itself can be defined by the equations 
$$
y^{p} - y = x \qquad\text{and}\qquad z^{p} - z = \frac{x^{p} + y^{p} - (x + y)^{p}}{p} + w
$$
where both $x$ and $w$ lie in $k$. Compare this to a $\Z/p\Z\times\Z/p\Z$-extension, which can be written as a tower of $\Z/p\Z$-extensions in multiple ways. The fact that $L/k$ is cyclic is reflected in the above equations defining the extension. To see this explicitly, suppose $H = \Gal(M/k) \cong \la\sigma\ra$ where $|\sigma| = p$. Then $\sigma$ acts on $M = k[x]/(x^{p} - x - a)$ via $\sigma(x) = x + 1$. Moreover, $\sigma$ generates $G = \Gal(L/k)$ if and only if 
$$
k[y,z]/(z^{p} - z - b) = k[y,z]/(z^{p} - z - \sigma(b))
$$
which in turn is equivalent to having $\sigma(b) = b + \wp(b')$ for some $b'\in M$. It's easy to see that when $L/k$ is Galois of order $p^{2}$ and factors as the tower above, then $\sigma(b) \equiv b \mod{\wp(M)}$ occurs precisely when $G \cong \Z/p\Z\times\Z/p\Z$, while $\sigma(b)\not\equiv b \mod{\wp(M)}$ coincides with the case $G \cong \Z/p^{2}\Z$.

\subsection{Artin--Schreier--Witt Theory}

For a general cyclic extension of order $p^{n}$, Artin--Schreier--Witt theory and the arithmetic of Witt vectors encode the above automorphism data in a systematic way. The basic theory can be found in various places, including \cite[p.~330]{lang}.

Assume that $k$ is a field of characteristic $p > 0$ and $A$ is a $k$-algebra. Let $\W(A)$ be the ring of $p$-typical Witt vectors over $A$ and for each $n\geq 1$, let $\W_{n}(A)$ be the ring of Witt vectors of length $n$ over $A$, i.e.~the image of the ring homomorphism 
\begin{align*}
  t_{n} : \W(A) &\longrightarrow \W(A)\\
  (a_{0},\ldots,a_{n - 1},a_{n},a_{n + 1},\ldots) &\longmapsto (a_{0},\ldots,a_{n - 1},0,0,\ldots). 
\end{align*}
We write an element of $\W_{n}(A)$ as $(a_{0},\ldots,a_{n - 1})$. The Verschiebung operator defined by $V : \W(A)\rightarrow\W(A),(a_{0},a_{1},\ldots)\mapsto (0,a_{0},a_{1},\ldots)$ is an abelian group homomorphism, and moreover, $\W_{n}(A) \cong \W(A)/V^{n}\W(A)$. 
%
On the other hand, the Frobenius operator $F : \W(A)\rightarrow\W(A)$, defined by $(a_{0},a_{1},\ldots)\mapsto (a_{0}^{p},a_{1}^{p},\ldots)$, is a ring homomorphism generalizing the usual Frobenius on $A$. 


For $A = k$, set $\wp = F - \operatorname{id} : x\mapsto Fx - x$. Then $\wp$ is an abelian group homomorphism $\W_{n}(k)\rightarrow\W_{n}(k)$ generalizing the isogeny $\G_{a}\rightarrow\G_{a},x\mapsto x^{p} - x$. 
%

\begin{ex}
When $k = \F_{p}$, the field of $p$ elements, we have an isomorphism 
\begin{align*}
  \W_{n}(\F_{p}) &\longrightarrow \Z/p^{n}\Z\\
  (x_{0},x_{1},\ldots,x_{n - 1}) &\longmapsto \bar{x}_{0} + p\bar{x}_{1} + \ldots + p^{n - 1}\bar{x}_{n - 1}
\end{align*}
for all $n\geq 1$, where $\bar{x}_{i}$ denotes the image of $x_{i}$ under the canonical surjection $\Z/p\Z\rightarrow \Z/p^{n}\Z$. These commute with the natural maps $\Z/p^{n}\Z\rightarrow\Z/p^{n + 1}\Z$, giving an isomorphism 
$$
\W(\F_{p}) \xrightarrow{\;\sim\;} \invlim \Z/p^{n}\Z = \Z_{p}. 
$$
\end{ex}

Suppose $x\in\W_{n}(k)$ and $\alpha\in\W_{n}(k^{sep})$ are Witt vectors such that $\wp(\alpha) = x$. If $\alpha = (\alpha_{0},\ldots,\alpha_{n - 1})$, we write $k(\wp^{-1}x) = k(\alpha_{0},\ldots,\alpha_{n - 1})$ as a field extension of $k$. The following theorem characterizes cyclic extensions of degree $p^{n}$ of $k$. 

\begin{thm}
\label{thm:AWSextclassification}
Let $k$ be a field of characteristic $p > 0$. Then for each $n\geq 1$, there is a bijection 
\begin{align*}
  \left\{\begin{matrix} \text{cyclic extensions } L/k \text{ with} \\ [L : k] = p^{n}\end{matrix}\right\} &\longleftrightarrow \W_{n}(k)/\wp(\W_{n}(k))\\
  L = k(\wp^{-1}x) &\longleftrightarrow x. 
\end{align*}
\end{thm}

%

Alternatively, any cyclic extension $L/k$ with Galois group $G\cong\Z/p^{n}\Z$ can be given by a system of equations 
$$
y_{i}^{p} - y_{i} = f_{i}(f_{0},\ldots,f_{i - 1};y_{0},\ldots,y_{i - 1}) \text{ for } 0\leq i\leq n - 1
$$
where $f_{0}\in k$ and each $f_{i}$ is a polynomial over $k$. This follows from Artin--Schreier theory and the fact that a cyclic $\Z/p^{n}\Z$-extension can be written as a tower of $\Z/p\Z$-extensions in a unique way.

\subsection{Ramification Data}
\label{sec:ram}

Suppose $k$ is a complete local field of characteristic $p$. Then $k\cong k_{0}((t))$ for an algebraically closed field $k_{0}$ and the Galois theory of $k$ can be described by filtering the Galois group $G = \Gal(L/k)$ of any separable extension according to liftings of the $t$-adic valuation to $L$. In particular, let $\orb_{k}$ be the valuation ring of $k$, or equivalently, the subring of $k$ corresponding to $k_{0}[[t]]$. It contains a prime ideal $\frak{p}_{k}$ corresponding to $(t)\subset k_{0}[[t]]$. For any separable extension $L/k$, let $\orb_{L}$ be the valuation ring of $L$, which can be defined as the integral closure of $\orb_{k}$ in $L$. The unique prime ideal lying over $\frak{p}_{k}$ will be denoted $\frak{P}_{L}$. 

The Galois group $G$ contains subgroups 
$$
I = \{\sigma\in G : \sigma(x)\equiv x\mod{\frak{P}_{L}} \text{ for all } x\in\orb_{L}\},
$$
called the {\it inertia group} of $L/k$, and 
$$
R = \left\{\sigma\in G : \frac{\sigma(x)}{x}\equiv 1\mod{\frak{P}_{L}} \text{ for all } x\in L^{\times}\right\},
$$
called the {\it ramification group}. These form the start of a filtration of the Galois group: $G\supseteq I\supseteq R$. For each $i\geq 0$, define 
$$
G_{i} = \{\sigma\in G : v_{L}(\sigma(x) - x)\geq i + 1 \text{ for all } x\in\orb_{L}\},
$$
where $v_{L}$ denotes the unique extension of the $t$-adic valuation to $L$. Then $G_{0} = I,G_{1} = R$ and we get a filtration of $G$ by normal subgroups: 
$$
G\supseteq G_{0}\supseteq G_{1}\supseteq G_{2}\supseteq\cdots
$$
This is called the {\it ramification filtration in the lower numbering} for $G$; it terminates in a finite number of steps. If $G_{m}\supsetneq G_{m + 1}$, $m$ is called a {\it jump} in the ramification filtration. It is known \cite[Ch.~IV]{ser2} (and see Proposition~\ref{prop:ramfilt} below) that $G_{0}$ is a semidirect product of the form $P\rtimes\Z/r\Z$, where $P$ is a $p$-group, say of order $p^{n}$, and $r$ is prime to $p$. Moreover, $G_{1}$ is the unique Sylow $p$-subgroup of $G_{0}$, so there are exactly $n$ jumps in the ramification filtration. 

A parallel filtration of $G$ can be defined as follows. Define a function $\varphi = \varphi_{L/k} : [0,\infty)\rightarrow [0,\infty)$ by 
$$
\varphi(i) = \frac{1}{|G_{0}|}(|G_{1}| + \ldots + |G_{m}| + (i - m)|G_{m + 1}|)
$$
for $m\in\Z$ with $m\leq i\leq m + 1$. (This is usually written as an integral; see {\it loc.~cit.}) Define the {\it ramification filtration in the upper numbering} for $G$ by 
$$
G\supseteq G^{0}\supseteq G^{1}\supseteq G^{2}\supseteq\cdots
$$
where $G^{j} = G_{i}$ for $j = \varphi(i)$. An easy formula for translating between the extensions is due to Herbrand \cite[Ch.~IV]{ser2}: if $m_{0} = u_{0} = 0$ and for $k\geq 1$, $m_{k}$ (resp.~$u_{k}$) are the ramification jumps in the lower (resp.~upper) numbering, then 
$$
u_{k} - u_{k - 1} = \frac{1}{p^{k - 1}r}(m_{k} - m_{k - 1}). 
$$

The filtration in the upper numbering is compatible with quotients of $G$ (subextensions of $L/k$), whereas the filtration in the lowering numbering is only compatible with subgroups of $G$. However, the jumps in the upper numbering need not be integers, though they are when $G$ is abelian \cite[Ch.~V, Sec.~7]{ser2}. 

Here we record some useful facts about the ramification filtrations of $G = \Gal(L/k)$. 

\begin{prop}[{\cite[Ch.~IV]{ser2}}, {\cite[Prop.~4.2]{op}}]
\label{prop:ramfilt}
For a Galois extension $L/k$ with group $G$, 
\begin{enumerate}[\quad (a)]
  \item $G_{0} \cong P\rtimes\Z/r\Z$ where $P$ is a finite $p$-group and $r$ is prime to $p$. 
  \item $G_{0}/G_{1}$ is cyclic of order $r$. 
  \item $G_{1}$ is the Sylow $p$-group of $G_{0}$. 
  \item For each $i\geq 1$, the quotient $G_{i}/G_{i + 1}$ is a direct product of cyclic groups of order $p$. 
  \item The jumps in the lower numbering are congruent mod $r$. 
  \item The jumps in the upper numbering are congruent mod $r$. 
\end{enumerate}
\end{prop}

We now turn to cyclic $p^{n}$-extensions. By Theorem~\ref{thm:AWSextclassification}, any $\Z/p^{n}\Z$-extension $L/k$ is of the form $L = k(\wp^{-1}x)$ for some Witt vector $x = (x_{0},x_{1},\ldots,x_{n - 1})\in\W_{n}(k)$. Set $m_{i} = -v(x_{i})$ for $0\leq i\leq n - 1$. 

\begin{thm}
\label{thm:garjump}
The last jump in the ramification filtration in the upper numbering for $G = \Gal(L/k)$ is $u = \max\{p^{n - i}m_{i}\}_{i = 0}^{n - 1}$. 
\end{thm}

\begin{proof}
This follows from \cite[Thm.~1.1]{gar}. A proof using local class field theory can be found in \cite[Sec.~5]{tho}. 
\end{proof}

Therefore the ramification filtration (either in the upper or lower numbering) of a cyclic $\Z/p^{n}\Z$-extension of complete local fields can be determined completely by its Witt vector equation. For further reading, in the last section of \cite{op} the authors provide explicit equations describing $\Z/p^{3}\Z$-equations of $k((t))$.


\section{Artin--Schreier--Witt Root Stacks}
\label{sec:ASWroots}

Let $k$ be a field of characteristic $p > 0$ and let $L/k$ be Galois extension with Galois group $G = \Z/p^{n}\Z$ for some $n\geq 1$. By Theorem~\ref{thm:AWSextclassification}, such an extension is of the form 
$$
L = k[x]/(\wp x - a)
$$
where $x = (x_{0},\ldots,x_{n - 1})$ is an indeterminate taking values in the ring of length $n$ Witt vectors $\W_{n}(k)$ and $a\in\W_{n}(k)$ is not of the form $a = \wp b$ for any $b\in\W_{n}(k)$. For $n = 1$, $\wp$ is the just map $\alpha\mapsto\alpha^{p} - \alpha$, used in~\cite[Sec.~6]{kob} to construct the universal Artin--Schreier covers 
$$
\wp_{m} : [\P(1,m)/\G_{a}] \longrightarrow [\P(1,m)/\G_{a}]. 
$$
These were used to define Artin--Schreier root stacks, the wild $\Z/p\Z$ analogue of tame root stacks. See Section~\ref{sec:ASroots} for a brief review. 

To study higher order wild root stacks, we will replace the quotient stack $[\P(1,m)/\G_{a}]$ with $[\overline{\W}_{n}(1,m_{1},\ldots,m_{n})/\W_{n}]$, where $\overline{\W}_{n}(1,m_{1},\ldots,m_{n})$ is a new stacky equivariant compactification of $\W_{n}$ equal to $\P(1,m)$ in the $n = 1$ case. This stacky compactification is built on a compactification $\overline{\W}_{n}$ of $\W_{n}$ in the category of schemes, due to Garuti \cite{gar}, which we describe in Section~\ref{sec:garuti}.

\subsection{Artin--Schreier Root Stacks}
\label{sec:ASroots}

Fix a prime $p$ and an integer $m\geq 1$. As above, the universal Artin--Schreier cover for this pair $(p,m)$ is the morphism 
$$
\wp_{m} : [\P(1,m)/\G_{a}] \longrightarrow [\P(1,m)/\G_{a}]
$$
induced by the compatible maps $[u : v]\mapsto [u^{p} : v^{p} - vu^{m(p - 1)}]$ on $\P(1,m)$ and $\alpha\mapsto \alpha^{p} - \alpha$ on $\G_{a}$. Locally, points of $[\P(1,m)/\G_{a}]$ are triples $(L,s,f)$, where $L$ is a line bundle, $s$ is a section of $L$ and $f$ is a section of $L^{m}$, with disjoint zero sets. Pulling back along $\wp_{m}$ takes an Artin--Schreier root of the triple $(L,s,f)$ as follows. For a scheme $X$ and a triple $(L,s,f)$ corresponding to a map $X\rightarrow [\P(1,m)/\G_{a}]$, the {\it Artin--Schreier root} of $X$ over $\div(s)$ with jump $m$ is the normalized pullback 
\begin{center}
\begin{tikzpicture}[xscale=4,yscale=2]
  \node at (0,1) (a) {$\wp_{m}^{-1}((L,s,f)/X)$};
  \node at (1,1) (b) {$[\P(1,m)/\G_{a}]$};
  \node at (0,0) (c) {$X$};
  \node at (1,0) (d) {$[\P(1,m)/\G_{a}]$};
  \draw[->] (a) -- (b);
  \draw[->] (a) -- (c);
  \draw[->] (b) -- (d) node[right,pos=.5] {$\wp_{m}$};
  \draw[->] (c) -- (d);
  \node at (.15,.7) {$\nu$};
  \draw (.1,.6) -- (.2,.6) -- (.2,.8);
\end{tikzpicture}
\end{center}

Properties of this construction are summarized below; see \cite[Sec.~6]{kob}. 

\begin{prop}
Let $X$ be a scheme and $(L,s,f)$ be a triple on $X$ corresponding to a morphism $X\rightarrow[\P(1,m)/\G_{a}]$. Then 
\begin{enumerate}[\quad (a)]
  \item Artin--Schreier roots are functorial. That is, for any morphism $\varphi : Y\rightarrow X$, there is an isomorphism of stacks 
  $$
  \wp_{m}^{-1}(\varphi^{*}(L,s,f)/Y) \cong \wp_{m}^{-1}((L,s,f)/X)\times_{X}^{\nu}Y. 
  $$
  \item If $X$ is a scheme over a perfect field $k$, $\wp_{m}^{-1}((L,s,f)/X)$ is a Deligne--Mumford stack with coarse space $X$. 
  \item Locally in the \'{e}tale topology, $\wp_{m}^{-1}((L,s,f)/X)$ is isomorphic to a quotient of the form $[V/G]$ where $G = \Z/p\Z$ and $V$ is an Artin--Schreier cover of (an \'{e}tale neighborhood of) $X$. 
\end{enumerate}
\end{prop}

\subsection{Garuti's Compactification}
\label{sec:garuti}

For a vector bundle $E\rightarrow X$, let $\P(E)\rightarrow X$ denote the {\it projective bundle} associated to $E$, that is, $\P(E) = \Proj_{X}(\Sym(E^{\vee}))$. This comes equipped with a tautological bundle $\orb_{\P}(1)$. Set $\orb_{\P}(m) = \orb_{\P}(1)^{\otimes m}$ for any $m\in\Z$, where $\orb_{\P}(-1) = \orb_{\P}(1)^{\vee}$ by convention. 

Following \cite{gar}, we define a sequence of ringed spaces $(\overline{\W}_{n},\orb_{\overline{\W}_{n}}(1))$ inductively by 
\begin{align*}
  (\overline{\W}_{1},\orb_{\overline{\W}_{1}}(1)) &= (\P^{1},\orb_{\P^{1}}(1))\\
  \text{and}\quad (\overline{\W}_{n},\orb_{\overline{\W}_{n}}(1)) &= (\P(\orb_{\overline{\W}_{n - 1}}\oplus\orb_{\overline{\W}_{n - 1}}(p)),\orb_{\P}(1)) \quad\text{for } n\geq 2,
\end{align*}
where $\orb_{\P}(1)$ is the tautological bundle of the projective bundle in that step. There is a morphism 
$$
r : \overline{\W}_{n} \longrightarrow \overline{\W}_{n - 1}
$$
for all $n\geq 1$ exhibiting $\overline{\W}_{n}$ as a $\P^{1}$-bundle over $\overline{\W}_{n - 1}$. Note that $r_{*}\orb_{\overline{\W}_{n}}(1) = \orb_{\overline{\W}_{n - 1}}\oplus\orb_{\overline{\W}_{n - 1}}(p)$. For each $n\geq 2$, there is a canonical section of $r$ corresponding to the zero section of the bundle $\P(\orb_{\overline{\W}_{n - 1}}\oplus\orb_{\overline{\W}_{n - 1}}(p))$ over $\overline{\W}_{n - 1}$. Let $Z_{n}$ be the divisor associated to the zero locus of this section. On the other hand, the isomorphism 
$$
\P(\orb_{\overline{\W}_{n - 1}}\oplus\orb_{\overline{\W}_{n - 1}}(p)) \cong \P(\orb_{\overline{\W}_{n - 1}}(-p)\oplus\orb_{\overline{\W}_{n - 1}})
$$
induces another section of $r$, called the ``infinity section'', whose divisor (aka zero locus) we denote by $\Sigma_{n}$. 

\begin{prop}[{\cite[Prop.~2.4]{gar}}]
\label{prop:boundarydivisorsWn}
There is a system of open immersions $j_{n} : \W_{n}\hookrightarrow\overline{\W}_{n}$ such that $j_{n}(\W_{n}) = \overline{\W}_{n}\smallsetminus B_{n}$ where $B_{n}$ is the zero locus of a section of $\orb_{\overline{\W}_{n}}(1)$, given by 
$$
B_{1} = \Sigma_{1} \quad\text{and}\quad B_{n} = \Sigma_{n} + pr^{*}B_{n - 1} \text{ for } n\geq 2. 
$$
\end{prop}

\begin{cor}[{\cite[Cor.~2.5]{gar}}]
For all $n\geq 2$, 
$$
B_{n} = \sum_{i = 1}^{n} p^{n - i}(r^{n - i})^{*}\Sigma_{i}. 
$$
\end{cor}

We next observe that $\overline{\W}_{n}$ is a compactification of $\W_{n}$ which is equivariant with respect to the action of $\W_{n}$ on itself. 

\begin{lem}[{\cite[Lem.~2.7]{gar}}]
\label{lem:Wngraded}
Let $\orb_{\overline{\W}_{n}}(1)$ be the tautological bundle on $\overline{\W}_{n}$. Then 
\begin{enumerate}[\quad (1)]
  \item $\orb_{\overline{\W}_{n}}(1)$ is generated by global sections. 
  \item For any $m\geq 0$, there is an isomorphism of rings 
  $$
  H^{0}(\overline{\W}_{n},\orb_{\overline{\W}_{n}}(m)) \xrightarrow{\;\sim\;} \Sym^{m}(H_{p^{n - 1}})
  $$
  where $H_{d}$ denotes the $d$th graded piece of the graded ring 
  $$
  H = \F_{p}[t,y_{0},y_{1},\ldots,]. 
  $$
  \item Under this isomorphism, $Y_{n - 1}$ and $T^{p^{n - 1}}$ define principal divisors 
  $$
  (Y_{n - 1}) = \sum a_{P}P \qquad\text{and}\qquad (T^{p^{n - 1}}) = \sum b_{P}P
  $$
  such that $\sum_{a_{P}\geq 0} a_{P}P = Z_{n}$ and $\sum_{b_{P}\geq 0} b_{P}P = B_{n}$. 
\end{enumerate}
\end{lem}

This allows us to construct the action of $\W_{n}$ on $\overline{\W}_{n}$. 

\begin{prop}
The action of $\W_{n}$ on itself by Witt-vector translation extends to an action on $\overline{\W}_{n}$ which stabilizes $\orb_{\overline{\W}_{n}}(1)$. 
\end{prop}

\begin{proof}
(Sketch) For $n = 1$, the translation action of $\W_{1} =\G_{a}$ on itself by $\lambda\cdot x = x + \lambda$ extends to an action on $\P^{1} = \overline{\W}_{1}$ by $\lambda\cdot [x,y] = [x + \lambda y,y]$. Since this fixes $\infty = [1 : 0]$, the action stabilizes $\orb(1) = \orb(1\cdot\infty)$. The general case is proved by induction \cite[Prop.~2.8]{gar}. 
\end{proof}

\begin{prop}
The isogeny $\wp : \W_{n}\rightarrow\W_{n}$ extends to a cyclic cover of degree $p^{n}$, 
$$
\Psi_{n} : \overline{\W}_{n} \longrightarrow \overline{\W}_{n}
$$
which is defined over $\F_{p}$, commutes with the maps $r : \overline{\W}_{n}\rightarrow\overline{\W}_{n - 1}$ and has branch locus $B_{n}$, with $\Psi_{n}^{*}B_{n} = pB_{n}$. 
\end{prop}

\begin{proof}
(Sketch) The $n = 1$ case is well-known and is also outlined in~\cite[Sec.~6]{kob}. To induct, consider the fibre product 
\begin{center}
\begin{tikzpicture}[xscale=2.5,yscale=2]
  \node at (0,1) (a) {$P$};
  \node at (1,1) (b) {$\overline{\W}_{n + 1}$};
  \node at (0,0) (c) {$\overline{\W}_{n}$};
  \node at (1,0) (d) {$\overline{\W}_{n}$};
  \draw[->] (a) -- (b) node[above,pos=.5] {$\pi$};
  \draw[->] (a) -- (c) node[left,pos=.5] {$q$};
  \draw[->] (b) -- (d) node[right,pos=.5] {$r$};
  \draw[->] (c) -- (d) node[above,pos=.5] {$\Psi_{n}$};
\end{tikzpicture}
\end{center}
Then $\pi : P\rightarrow\overline{\W}_{n + 1}$ is a cyclic $p^{n}$-cover given explicitly by 
$$
P = \P(\orb_{\overline{\W}_{n}},\orb_{\overline{\W}_{n}}(p^{2}))
$$
since $\Psi_{n}^{*}B_{n} = pB_{n}$ and $\orb_{\P}(B_{n}) = \orb_{\overline{\W}_{n}}(p)$. Using Lemma~\ref{lem:Wngraded}, it is possible to construct a finite, flat morphism 
$$
\varphi : \overline{\W}_{n + 1} \longrightarrow P
$$
over $\overline{\W}_{n}$. 
This defines $\Psi_{n + 1}$ as the composition 
$$
\Psi_{n + 1} : \overline{\W}_{n + 1} \xrightarrow{\varphi} P \xrightarrow{\pi} \overline{\W}_{n + 1} 
$$
which is then finite, flat and extends $\wp : \W_{n + 1}\rightarrow\W_{n + 1}$ by construction. It is easy to check (cf.~\cite[Prop.~2.9]{gar}) that the $\Psi_{n}$ and $r$ commute. Finally, (3) of Lemma~\ref{lem:Wngraded} tells us that $B_{n + 1}$ is the effective part of the principal divisor $(t^{p^{n}})$, so $\Psi_{n + 1}^{*}B_{n} = pB_{n + 1}$ follows from the fact that $\Psi_{1}^{*}(t) = (t^{p})$. 
\end{proof}

\begin{rem}
By construction, $\overline{\W}_{2}$ can be identified with the Hirzebruch surface $F_{p}$. More generally, the sequence of $\overline{\W}_{n}$ form a {\it Bott tower} \cite{gk}. In particular, each $\overline{\W}_{n}$ is a smooth, projective toric variety \cite{cs}. 
\end{rem}

\subsection{Artin--Schreier--Witt Root Stacks}
\label{sec:AWSrootstacks}

Next we turn to the construction of the stacky compactification $\overline{\W}_{n}(1,m_{1},\ldots,m_{n})$ of the Witt scheme $\W_{n}$ for $n > 1$. We begin by setting $\overline{\W}_{1}(1,m) := \P(1,m)$, our stacky compactification of $\W_{1} = \A^{1}$. The key insight for generalizing this is to use the fact \cite[Lem.~6.3]{kob} that $\P(1,m)$ is itself a root stack over $\P^{1}$: 
\begin{center}
\begin{tikzpicture}[xscale=3.5,yscale=2]
  \node at (0,1) (a) {$\P(1,m)$};
  \node at (1,1) (b) {$[\A^{1}/\G_{m}]$};
  \node at (0,0) (c) {$\P^{1}$};
  \node at (1,0) (d) {$[\A^{1}/\G_{m}]$};
  \draw[->] (a) -- (b);
  \draw[->] (a) -- (c);
  \draw[->] (b) -- (d) node[right,pos=.5] {$m$};
  \draw[->] (c) -- (d) node[above,pos=.5] {$(\orb_{\P^{1}}(1),\Sigma_{1})$};
\end{tikzpicture}
\end{center}
Pulling back $\P(1,m) = \overline{\W}_{1}(1,m)$ along the sequence 
$$
\cdots\rightarrow \overline{\W}_{3}\xrightarrow{r} \overline{\W}_{2}\xrightarrow{r} \overline{\W}_{1} = \P^{1}
$$
defines $\overline{\W}_{n}(1,m,1,\ldots,1)$ for each $n > 1$. Each of these is a root stack over $\overline{\W}_{n}$ with stacky structure at (the pullback of) $\Sigma_{1}$; for example, $\overline{\W}_{2}(1,m,1) = r^{*}\overline{\W}_{1}(1,m)$ is a root stack over $\overline{\W}_{2}$: 
\begin{center}
\begin{tikzpicture}[xscale=4,yscale=2]
  \node at (0,1) (a) {$\overline{\W}_{2}(1,m,1)$};
  \node at (1,1) (b) {$[\A^{1}/\G_{m}]$};
  \node at (0,0) (c) {$\overline{\W}_{2}$};
  \node at (1,0) (d) {$[\A^{1}/\G_{m}]$};
  \draw[->] (a) -- (b);
  \draw[->] (a) -- (c);
  \draw[->] (b) -- (d) node[right,pos=.5] {$m$};
  \draw[->] (c) -- (d) node[above,pos=.5] {$(r^{*}\orb_{\P^{1}}(1),r^{*}\Sigma_{1})$};
\end{tikzpicture}
\end{center}
For a pair of positive integers $(m_{1},m_{2})$, the compactification $\overline{\W}_{2}(1,m_{1},m_{2})$ of $\W_{2}$ is defined by a second root stack, $\overline{\W}_{2}(1,m_{1},m_{2}) := \sqrt[m_{2}]{(\orb(1),\Sigma_{2})/\overline{\W}_{2}(1,m_{1},1)}$: 
\begin{center}
\begin{tikzpicture}[xscale=3.8,yscale=2]
  \node at (0,1) (a) {$\overline{\W}_{2}(1,m_{1},m_{2})$};
  \node at (1,1) (b) {$[\A^{1}/\G_{m}]$};
  \node at (0,0) (c) {$\overline{\W}_{2}(1,m_{1},1)$};
  \node at (1,0) (d) {$[\A^{1}/\G_{m}]$};
  \draw[->] (a) -- (b);
  \draw[->] (a) -- (c);
  \draw[->] (b) -- (d) node[right,pos=.5] {$m_{2}$};
  \draw[->] (c) -- (d) node[above,pos=.5] {$(\orb(1),\Sigma_{2})$};
\end{tikzpicture}
\end{center}
Here, $\orb(1)$ denotes the pullback of the line bundle $\orb_{\overline{\W}_{2}}(1)$ to $\overline{\W}_{2}(1,m_{1},1)$ along the coarse map. Now we proceed inductively. Let $n\geq 2$. 

\begin{defn}
For a sequence of positive integers $(m_{1},\ldots,m_{n})$, define the {\bf compactified Witt stack} $\overline{\W}_{n}(1,m_{1},\ldots,m_{n})$ to be the root stack $\sqrt[m_{n}]{(\orb(1),\Sigma_{n})/\overline{\W}_{n}(1,m_{1},\ldots,m_{n - 1},1)}$: 
\begin{center}
\begin{tikzpicture}[xscale=5,yscale=2]
  \node at (0,1) (a) {$\overline{\W}_{n}(1,m_{1},\ldots,m_{n})$};
  \node at (1,1) (b) {$[\A^{1}/\G_{m}]$};
  \node at (0,0) (c) {$\overline{\W}_{n}(1,m_{1},\ldots,m_{n - 1},1)$};
  \node at (1,0) (d) {$[\A^{1}/\G_{m}]$};
  \draw[->] (a) -- (b);
  \draw[->] (a) -- (c);
  \draw[->] (b) -- (d) node[right,pos=.5] {$m_{n}$};
  \draw[->] (c) -- (d) node[above,pos=.5] {$(\orb(1),\Sigma_{n})$};
\end{tikzpicture}
\end{center}
where $\overline{\W}_{n}(1,m_{1},\ldots,m_{n - 1},1) = r^{*}\overline{\W}_{n - 1}(1,m_{1},\ldots,m_{n - 1})$ is the pullback along $r$ of the compactified Witt stack $\overline{\W}_{n - 1}(1,m_{1},\ldots,m_{n - 1})$ over $\overline{\W}_{n - 1}$, and $\orb(1)$ is pulled back inductively as explained above. 
\end{defn}

We will continue to abuse notation by writing $r$ for the natural projections 
$$
\overline{\W}_{n}(1,m_{1},\ldots,m_{n})\rightarrow\overline{\W}_{n - 1}(1,m_{1},\ldots,m_{n - 1}).
$$

\begin{prop}
\label{prop:univASWcoverstack}
For each $n\geq 1$, the cyclic $p^{n}$-cover $\Psi_{n} : \overline{\W}_{n}\rightarrow\overline{\W}_{n}$ extends to a morphism of stacks 
$$
\Psi = \Psi_{m_{1},\ldots,m_{n}} : \overline{\W}(m_{1},\ldots,m_{n})\rightarrow\overline{\W}(m_{1},\ldots,m_{n})
$$
which commutes with $r : \overline{\W}(m_{1},\ldots,m_{n})\rightarrow\overline{\W}(m_{1},\ldots,m_{n - 1})$ and satisfies $\Psi^{*}B_{n} = pB_{n}$. 
\end{prop}

\begin{proof}
For $n = 1$, $\Psi_{1} : \P^{1}\rightarrow\P^{1}$ is the extension of $\wp(x) = x^{p} - x$ from $\A^{1}$ to $\P^{1}$. As explained in \cite[Sec.~6]{kob}, this extends naturally to $\overline{\W}_{1}(1,m) = \P(1,m)$ as $[x,y]\mapsto [x^{p},y^{p} - yx^{m(p - 1)}]$. Then by construction $\Psi^{*}\Sigma_{1} = p\Sigma_{1}$. To induct, suppose $\Psi : \overline{\W}_{n - 1}(1,m_{1},\ldots,m_{n - 1})\rightarrow\overline{\W}_{n - 1}(1,m_{1},\ldots,m_{n - 1})$ has been constructed. Then pulling back along $r$ extends $\Psi$ to a cover $\overline{\W}_{n}(1,m_{1},\ldots,m_{n - 1},1)\rightarrow\overline{\W}_{n}(1,m_{1},\ldots,m_{n - 1},1)$. Since the root stack construction commutes with pullback (cf.~\cite[Rem.~2.2.3]{cad} or~\cite[Lem.~5.10]{kob}), this induces a morphism $\Psi : \overline{\W}_{n}(1,m_{1},\ldots,m_{n})\rightarrow\overline{\W}_{n}(1,m_{1},\ldots,m_{n})$. By construction this commutes with $r : \overline{\W}_{n}(1,m_{1},\ldots,m_{n})\rightarrow\overline{\W}_{n - 1}(1,m_{1},\ldots,m_{n - 1})$ and we can compute 
\begin{align*}
  \Psi^{*}B_{n} &= \Psi^{*}(\Sigma_{n} + pr^{*}B_{n - 1}) \quad\text{by Proposition~\ref{prop:boundarydivisorsWn}}\\
    &= \Psi^{*}\Sigma_{n} + pr^{*}(\Psi^{*}B_{n - 1}) \quad\text{since $\Psi$ and $r$ commute}\\
    &= p\Sigma_{n} + pr^{*}(pB_{n - 1}) \quad\text{by induction}\\
    &= p(\Sigma_{n} + pr^{*}B_{n - 1}) = pB_{n}. 
\end{align*}
\end{proof}

Next, we prove a generalization of \cite[Prop.~6.4]{kob} that describes the $T$-points of the stack $\overline{\W}_{n}(1,m_{1},\ldots,m_{n})$ for any scheme $T$ in terms of line bundles on $T$ and their sections. First, we reinterpret $\overline{\W}_{n}(1,m_{1},\ldots,m_{n})$ as a quotient stack. 

\begin{lem}
\label{lem:Wittstackquotient}
For each $n\geq 2$ and any sequence of positive integers $(m_{1},\ldots,m_{n})$, the compactified Witt stack is a quotient stack: 
$$
\overline{\W}_{n}(1,m_{1},\ldots,m_{n}) = [V_{n}\smallsetminus\{0\}/\G_{m}]
$$
where $V_{n}$ is the total space of a rank $2$ vector bundle $E_{n}$ on $\overline{\W}_{n - 1}(1,m_{1},\ldots,m_{n - 1})$ and $\G_{m}$ acts on $V_{n}$ with weights $(1,m_{n})$. 
\end{lem}

\begin{proof}
(Sketch) First consider the case when $m_{1} = \cdots = m_{n} = 1$, that is $\overline{\W}_{n}(1,1,\ldots,1) = \overline{\W}_{n}$. By definition, $\overline{\W}_{n} = \P(E_{n})$ where $E_{n} = \orb_{\overline{\W}_{n - 1}}\oplus\orb_{\overline{\W}_{n - 1}}(p)$. For any vector bundle $E$ on $X$, the projective bundle $\P(E)$ can be presented as a quotient stack  
$$
\P(E) = [V\smallsetminus\{0\}/\G_{m}]
$$
where $V = \Sym(E^{*})$ is the total space of $E$ and $\G_{m}$ acts on $V$ by scalar multiplication. This finishes the description in the unweighted case. 

For the general case, the rank $2$ vector bundle is 
$$
E_{n} = \orb_{\overline{\W}_{n - 1}(1,m_{1},\ldots,m_{n - 1})}\oplus\orb_{\overline{\W}_{n - 1}(1,m_{1},\ldots,m_{n - 1})}(p)
$$
with total space $V_{n}$, and $\P(E_{n})$ is replaced by a \emph{weighted} relative $\Proj$, with weights $(1,m_{n})$. In this case the weighted relative $\Proj$ is identified with the quotient $\P(E_{n}) = [V_{n}\smallsetminus\{0\}/\G_{m}]$ where $\G_{m}$ acts on $V_{n}$ with weights $(1,m_{n})$. 
\end{proof}

Recall \cite[Prop.~6.4]{kob} that for each $m\geq 1$, a morphism into the weighted projective stack $\P(1,m)$ is equivalent to the data of a triple $(L,s,f)$ consisting of a line bundle $L$, a section $s$ of $L$ and another section $f$ of $L^{\otimes m}$ such that $s$ and $f$ do not vanish simultaneously. 

For two weights $m,n\geq 1$, consider the compactified Witt stack $\overline{\W}_{2}(1,m,n)$. To generalize \cite[Prop.~6.4]{kob}, let $\DIV^{[1,m,n]}$ be the category fibred in groupoids whose objects are tuples $(T,L,s,f,g)$, where $T$ is a scheme, $L$ is a line bundle on $T$, $s\in H^{0}(T,L)$, $f\in H^{0}(T,L^{m})$ not vanishing simultaneously with $s$, and $g\in H^{0}(T,L^{n})$, also not vanishing simultaneously with $s$. Morphisms are compatible morphisms of line bundles taking sections to sections. Then $\DIV^{[1,m,n]} \cong \overline{\W}_{2}(1,m,n)$, as explained below. 

Starting with the case $m = n = 1$, we have $\overline{\W}_{2} = \P(E_{2})$, where $E_{2} = \orb_{\P^{1}}\oplus\orb_{\P^{1}}(p)$ by definition. By the universal property of this Proj bundle, a morphism $T\rightarrow\overline{\W}_{2}$ is determined by a map $T\xrightarrow{\varphi}\P^{1}$ (hence a triple $(L,s,f)$) and a subbundle $L\subseteq\varphi^{*}E_{2}$, which in turn determines a section $g = \varphi^{*}s_{0}$ of $L$. In other words, $g$ is determined by the divisor $\Sigma_{2}$ in $\overline{\W}_{2}$. 

For any weights $m,n\geq 1$, \cite[Prop.~6.4]{kob} shows that $\P(1,m)$ can be identified with $\DIV^{[1,m]}$, the stack of tuples $(T,L,s,f)$. Thus there is a forgetful morphism $\DIV^{[1,m,n]}\rightarrow\DIV^{[1,m]}\cong\P(1,m)$. Pulling things back to $\overline{\W}_{2}$ along $r$, we see that $\overline{\W}_{2}(1,m,1)$ can similarly be identified with $\DIV^{[1,m,1]}$ -- with the second section coming from $\Sigma_{2}$ as above. Finally, $\overline{\W}_{2}(1,m,n)$ is defined as a root stack over $\overline{\W}_{2}(1,m,1)$ along this divisor $\Sigma_{2}$, with weight $n$. Then \cite[Prop.~5.3]{kob} allows us to identify $\overline{\W}_{2}(1,m,n)$ with $\DIV^{[1,m,n]}$. Explicitly, for $\varphi : T\rightarrow\overline{\W}_{2}(1,m,n)$, the tuple $(L,s,f,g)$ is given by $L = \varphi^{*}\orb(1)$, $s = \varphi^{*}r^{*}[0 : 1]$, $f = \varphi^{*}r^{*}[1 : 0]$ and $g = \varphi^{*}\Sigma_{2}$. This is summarized in the following commutative diagram, in which the square is cartesian. 
\begin{center}
\begin{tikzpicture}[xscale=2.6,yscale=2]
  \node at (0,0) (a) {$\overline{\W}_{2}$};
  \node at (1,0) (b) {$\P^{1}$};
  \node at (0,1) (c) {$\overline{\W}_{2}(1,m,1)$};
  \node at (-.9,1.05) {$\DIV^{[1,m,1]} \cong$};
  \node at (1,1) (d) {$\P(1,m)$};
  \node at (1.65,1.05) {$\cong \DIV^{[1,m]}$};
  \node at (0,2) (e) {$\overline{\W}_{2}(1,m,n)$};
  \node at (-.9,2.05) {$\DIV^{[1,m,n]} \cong$};
  \draw[->] (d) -- (b);
  \draw[->] (a) -- (b) node[above,pos=.5] {$r$};
  \draw[->] (c) -- (a);
  \draw[->] (c) -- (d) node[above,pos=.5] {$r$};
  \draw[->] (e) -- (c);
\end{tikzpicture}
\end{center}

Now we turn to the general case. For a sequence of positive integers $m_{1},\ldots,m_{n}$ and a scheme $T$, let $\DIV^{[1,m_{1},\ldots,m_{n}]}(T)$ be the category whose objects are tuples $(L,s,f_{1},\ldots,f_{n})$ with $L$ a line bundle on $T$, $s\in H^{0}(T,L)$ and $f_{i}\in H^{0}(T,L^{m_{i}})$ for each $1\leq i\leq n$ that don't vanish simultaneously with $s$. Morphisms $(L,s,f_{1},\ldots,f_{n})\rightarrow (L',s',f_{1}',\ldots,f_{n}')$ in $\DIV^{[1,m_{1},\ldots,m_{n}]}(T)$ are given by bundle isomorphisms $\varphi : L\rightarrow L'$ taking $s\mapsto s'$ and $f_{i}\mapsto f_{i}'$. Then $\DIV^{[1,m_{1},\ldots,m_{n}]}$ is a category fibred in groupoids over $\cat{Sch}_{k}$. The proof in the $n = 2$ case above generalizes easily to show: 

\begin{prop}
\label{prop:ASWlinebundlesectionclassification}
For any $m_{1},\ldots,m_{n}\geq 1$, there is an isomorphism of categories fibred in groupoids 
$$
\DIV^{[1,m_{1},\ldots,m_{n}]} \cong \overline{\W}_{n}(1,m_{1},\ldots,m_{n}). 
$$
\end{prop}

\begin{cor}
For any $m_{1},\ldots,m_{n}\geq 1$, $\DIV^{[1,m_{1},\ldots,m_{n}]}$ is a stack of dimension $n$. 
\end{cor}

\begin{defn}
Let $X$ be a scheme, $(m_{1},\ldots,m_{n})$ a sequence of positive integers and consider a tuple $(L,s,f_{1},\ldots,f_{n})$ consisting of a line bundle $L$ on $X$ and sections $s\in\Gamma(X,L)$ and $f_{i}\in\Gamma(X,L^{m_{i}})$, $1\leq i\leq n$, which do not vanish simultaneously. The {\bf Artin--Schreier--Witt root stack} of $\X$ along $(L,s,f_{1},\ldots,f_{n})$ is the normalized pullback $\Psi^{-1}((L,s,f_{1},\ldots,f_{n})/\X)$ of the diagram 
\begin{center}
\begin{tikzpicture}[xscale=6,yscale=2]
  \node at (0,1) (a) {$\Psi^{-1}((L,s,f_{1},\ldots,f_{n})/X)$};
  \node at (1,1) (b) {$[\overline{\W}_{n}(1,m_{1},\ldots,m_{n})/\W_{n}]$};
  \node at (0,0) (c) {$X$};
  \node at (1,0) (d) {$[\overline{\W}_{n}(1,m_{1},\ldots,m_{n})/\W_{n}]$};
  \draw[->] (a) -- (b);
  \draw[->] (a) -- (c);
  \draw[->] (b) -- (d) node[right,pos=.5] {$\Psi$};
  \draw[->] (c) -- (d);
  \node at (.1,.7) {$\nu$};
  \draw (.067,.6) -- (.133,.6) -- (.133,.8);
\end{tikzpicture}
\end{center}
where $\Psi$ is the cyclic degree $p^{n}$ morphism from Proposition~\ref{prop:univASWcoverstack} and the bottom row is induced by $(L,s,f_{1},\ldots,f_{n})$, following Proposition~\ref{prop:ASWlinebundlesectionclassification}. 
\end{defn}

As in \cite[Sec.~6]{kob}, this definition extends to a base which is a stack. For a stack $\X$, set $\DIV^{[1,m_{1},\ldots,m_{n}]}(\X) = \Hom_{\cat{Stacks}}(\X,\DIV^{[1,m_{1},\ldots,m_{n}]})$ and likewise set $\overline{\W}_{n}(1,m_{1},\ldots,m_{n})(\X) = \Hom_{\cat{Stacks}}(\X,\overline{\W}_{n}(1,m_{1},\ldots,m_{n}))$. 

\begin{defn}
For a stack $\X$, a sequence of positive integers $(m_{1},\ldots,m_{n})$ and a tuple $(\L,s,f_{1},\ldots,f_{n})\in\overline{\W}_{n}(1,m_{1},\ldots,m_{n})(\X)$, the {\bf Artin--Schreier--Witt root stack} of $\X$ along $(\L,s,f_{1},\ldots,f_{n})$ is defined to be the normalized pullback $\Phi^{-1}((\L,s,f_{1},\ldots,f_{n})/\X)$ of the diagram 
\begin{center}
\begin{tikzpicture}[xscale=6,yscale=2]
  \node at (0,1) (a) {$\Psi^{-1}((\L,s,f_{1},\ldots,f_{n})/\X)$};
  \node at (1,1) (b) {$[\overline{\W}_{n}(1,m_{1},\ldots,m_{n})/\W_{n}]$};
  \node at (0,0) (c) {$\X$};
  \node at (1,0) (d) {$[\overline{\W}_{n}(1,m_{1},\ldots,m_{n})/\W_{n}]$};
  \draw[->] (a) -- (b);
  \draw[->] (a) -- (c);
  \draw[->] (b) -- (d) node[right,pos=.5] {$\Psi$};
  \draw[->] (c) -- (d);
  \node at (.1,.7) {$\nu$};
  \draw (.067,.6) -- (.133,.6) -- (.133,.8);
\end{tikzpicture}
\end{center}
\end{defn}

\begin{rem}
\label{rem:ASWlocalpts}
As in \cite[Rmk.~6.10]{kob}, we can interpret the $T$-points of an Artin--Schreier--Witt root stack $\Psi^{-1}((L,s,f_{1},\ldots,f_{n})/X)$ for ``local enough'' $T$: \'{e}tale-locally, they are tuples $(\varphi,M,t,g_{1},\ldots,g_{n},\psi)$ where $T\xrightarrow{\varphi}X$ is a morphism of schemes, $M$ is a line bundle on $T$, $M^{p^{n}}\xrightarrow{\psi}\varphi^{*}L$ is an isomorphism of line bundles, $t\in H^{0}(T,M)$ and for each $1\leq i\leq n$, $g_{i}\in H^{0}(T,M^{m_{i}})$, all satisfying 
$$
\psi(t^{p^{n}}) = \varphi^{*}s \quad\text{and}\quad \psi(g_{i}^{p} - t^{m_{i}(p - 1)}g_{i}) = \varphi^{*}f_{i} \text{ for } 1\leq i\leq n. 
$$

The global situation is a little more delicate than in {\it loc.~cit.}, so we take care to explain it here. Let $T$ be a normal scheme. For $n = 1$, the $T$-points of $\Psi^{-1}((L,s,f)/X)$ are tuples $(\varphi,M,t,g,\psi)$, this time with $g\in H^{0}(m_{1}(t),M^{m_{1}}|_{m_{1}(t)})$ a ``local section'', or germ at each point of the support of the divisor $m_{1}(t)$. Generalizing this, for any $n$, set $\X_{i} = \Psi^{-1}((L,s,f_{1},\ldots,f_{i})/X)$, $\eta_{i} : \X_{i}\rightarrow\X_{i - 1}$ the canonical projection, and $D_{i} = \eta_{i}^{-1}(t)$ for each $1\leq i\leq n - 1$. Then with $T$ still normal, the $T$-points of $\Psi^{-1}((L,s,f_{1},\ldots,f_{n})/X)$ are $(\varphi,M,t,g_{1},\ldots,g_{n},\psi)$ with $g_{i}\in H^{0}(m_{i - 1}D_{i - 1},M^{m_{i}}|_{D_{i - 1}})$ and the rest as above. A concrete example of this phenomenon can be found in Example~\ref{ex:elemASW}. When $T$ is not normal, things are probably too complicated to write down generally. However, a higher order version of \cite[Ex.~6.13]{kob} is possible in theory, either by iterating the method described in \cite[Rmk.~6.2]{kob} (see also \cite[Lem.~5.5]{ls}) or by generalizing that result using Witt vectors. See also \cite[Sec.~2]{mad}. 
\end{rem}


\section{Classification Theorems}
\label{sec:mainthms}

In this section, we use the construction of Artin--Schreier--Witt root stacks to classify stacky curves in positive characteristic with cyclic $p$th-power automorphism groups. This completes the cyclic version of the program begun in \cite{kob}. For the cyclic-by-tame case, see Subsection~\ref{sec:cyclicbytame}, and for remarks on the general case, see Section~\ref{sec:future}. 

\begin{lem}
\label{lem:functorial}
Let $h : \Y\rightarrow\X$ be a morphism of stacks and $(\L,s,f_{1},\ldots,f_{n})$ an object in $\DIV^{[1,m_{1},\ldots,m_{n}]}(\X)$. Then there is an isomorphism of algebraic stacks 
$$
\Psi^{-1}((h^{*}\L,h^{*}s,h^{*}f_{1},\ldots,h^{*}f_{n})/\Y) \xrightarrow{\;\sim\;} \Psi^{-1}((\L,s,f_{1},\ldots,f_{n})/\X)\times_{\X}^{\nu}\Y. 
$$
\end{lem}

\begin{proof}
See \cite[Lem.~6.11]{kob}. 
\end{proof}

\begin{ex}
\label{ex:key}
Consider the smooth, projective $\Z/p^{2}\Z$-cover $Y$ of $\P_{k}^{1}$ given birationally by the Witt vector equation $\wp\underline{x} = (t^{-j},0)$ where $\underline{x} = (x,y)\in\W_{2}(\bar{k})$ and $p\nmid j$. On the level of function fields, this corresponds to the tower of fields $L\supseteq K\supseteq k((t))$ with equations 
\begin{align*}
  x^{p} - x = t^{-j} &\qquad \text{(I)}\\
  y^{p} - y = t^{-j}x &\qquad \text{(II)}
\end{align*}
which has Galois groups $G = \Gal(L/k((t))) \cong \Z/p^{2}\Z$, $H = \Gal(L/K) \cong \Z/p\Z$ and $G/H = \Gal(K/k((t))) \cong \Z/p\Z$. Let $X$ be the smooth, projective curve with affine equation (I), giving us a sequence of covers $Y\xrightarrow{\psi} X\xrightarrow{\varphi} \P_{k}^{1}$. By Theorem~\ref{thm:garjump}, the ramification jumps in the upper numbering are $j$ and $pj$. If $\P_{k}^{1} = \Proj k[x_{0},x_{1}]$, \cite[Ex.~6.12]{kob} shows that the quotient stack $\X := [X/(G/H)]$ is an Artin--Schreier root stack over the point $[0 : 1]\in\P_{k}^{1}$ with jump $j$: 
$$
\X = [Y/(G/H)] \cong \wp_{j}^{-1}((\orb(1),x_{0},x_{1}^{j})/\P_{k}^{1}) \cong \P_{k}^{1}\times_{[\P(1,j)/\G_{a}]}^{\nu}[\P(1,j)/\G_{a}]. 
$$
Similarly, the quotient stack $\Zz := [Y/H]$ is an Artin--Schreier root stack over the preimage of $[0 : 1]$ in $X$, this time with jump $pj$: 
$$
\Zz = [Y/H] \cong \wp_{pj}^{-1}((\orb_{X}(1),s,f)/X) \cong X\times_{[\P(1,pj)/\G_{a}]}^{\nu}[\P(1,pj)/\G_{a}]
$$
where $s = \varphi^{*}x_{0}$ and $f\in H^{0}(\X,\orb_{\X}(pj)|_{P})$ corresponds to $t^{-j}x$ as a germ of a rational function at $P = \alpha^{-1}(\infty)$, where $\alpha : \X\rightarrow\P^{1}$ is the coarse moduli map. 

We'd like to describe $\Y := [Y/G]$ in a similar fashion. Below is a diagram showing the relations between $\P^{1},X,Y$ and the quotients $\X,\Y$ and $\Zz$: 
\begin{center}
\begin{tikzpicture}[xscale=2.8,yscale=1.2]
  \node at (0,2) (Y) {$Y$};
  \node at (0,1) (YH) {$[Y/H]$};
  \node at (0,0) (YG) {$[Y/G]$};
  \node at (1,1) (X) {$X$};
  \node at (1,0) (XGH) {$[X/(G/H)]$};
  \node at (2,0) (P1) {$\P^{1}$};
  \draw[->] (Y) -- (YH);
  \draw[->] (Y) -- (X);
  \draw[->] (YH) -- (YG);
  \draw[->,dashed] (YH) -- (X) node[above,pos=.4] {$\delta$};
  \draw[->] (X) -- (XGH);
  \draw[->] (X) -- (P1);
  \draw[->,dashed] (YG) -- (XGH) node[above,pos=.5] {$\beta$};
  \draw[->,dashed] (YG) to[out=-40,in=220] (P1);
    \node at (.9,-.8) {$\gamma = \alpha\circ\beta$};
  \draw[->,dashed] (XGH) -- (P1) node[above,pos=.4] {$\alpha$};
\end{tikzpicture}
\end{center}
Here, each solid vertical arrow is a degree $p$ quotient, the solid diagonal arrows are the degree $p$ ramified covers described above (equations (I) and (II)), and the dashed horizontal arrows are Artin--Schreier root stacks -- from Lemma~\ref{lem:doublequotient} it follows that $\beta$ is an Artin--Schreier root over the preimage of $\infty$ in $\X$ with jump $pj$, while the others are as described above. The composition $\gamma = \alpha\circ\beta$ can similarly be described as an Artin--Schreier--Witt root over $\infty$ with jumps $j$ and $pj$: the tuple $(L,x_{0},x_{1}^{j},f)$ on $\P^{1}$ determines a morphism $\P^{1}\rightarrow [\overline{\W}_{2}(1,j,pj)/\W_{2}]$ and pulling back along $\Psi = \Psi_{j,pj} : [\overline{\W}_{2}(1,j,pj)/\W_{2}]\rightarrow [\overline{\W}_{2}(1,j,pj)/\W_{2}]$ yields $\Y$: 
\begin{center}
\begin{tikzpicture}[xscale=4,yscale=2]
  \node at (0,0) (P1) {$\P^{1}$};
  \node at (0,1) (X) {$\X$};
  \node at (0,2) (Y) {$\Y$};
  \node at (1,0) (a1) {$[\P(1,j)/\G_{a}]$};
  \node at (1,1) (a2) {$[\P(1,j)/\G_{a}]$};
  \node at (2,0) (w1) {$[\overline{\W}_{2}(1,j,pj)/\W_{2}]$};
  \node at (2,2) (w2) {$[\overline{\W}_{2}(1,j,pj)/\W_{2}]$};
  \draw[->] (P1) -- (a1);
  \draw[->] (w1) -- (a1) node[above,pos=.5] {$r$};
  \draw[->] (X) -- (P1);
  \draw[->] (X) -- (a2);
  \draw[->] (w2) -- (a2) node[above,pos=.5] {$r$};
  \draw[->] (a2) -- (a1) node[left,pos=.5] {$\wp_{j}$};
  \draw[->] (w2) -- (w1) node[left,pos=.5] {$\Psi_{j,pj}$};
  \draw[->] (Y) -- (X);
  \draw[->] (Y) -- (w2);
\end{tikzpicture}
\end{center}
\end{ex}

\begin{ex}
\label{ex:elemASW}
More generally, for any curve $X$ and Witt vector-valued function $w\in \W_{n}(k(X))\smallsetminus\wp(\W_{n}(k(X)))$, let $Y_{w}$ be the curve over $X$ assigned to $F$ by Theorem~\ref{thm:AWSextclassification}; call the corresponding ramified $\Z/p^{n}\Z$-cover $\pi : Y_{w}\rightarrow X$. This determines a system of equations 
$$
y_{i}^{p} - y_{i} = F_{i}, \quad 0\leq i\leq n - 1
$$
where $F_{0}\in k(X)$ and each $F_{i}$ is a polynomial in $F_{0},\ldots,F_{i - 1},y_{0},\ldots,y_{i - 1}$ over $k(X)$. Then, \'{e}tale-locally about each ramification point on $X$, there is an isomorphism 
$$
\varphi : \Psi^{-1}((L,s,f_{1},\ldots,f_{n})/X) \xrightarrow{\;\sim\;} [Y_{w}/(\Z/p^{n}\Z)]
$$
where $(L,s,f_{1},\ldots,f_{n})$ is defined as follows. First, the pair $(L,s)$ corresponds to the divisor $\div(F_{0})$ on $X$. Next, for each $1\leq i\leq n - 1$, define $\X_{i}$ to be the stack obtained by replacing an \'{e}tale neighborhood $U_{P}$ of each point $P$ in the support of $\div(F_{0})$ with the quotient $[[U_{P}/G_{P,i}]/(G_{P,0}/G_{P,i})]$, where $G_{P,0} = \Gal(U_{P}/\pi(U_{P}))$ and $G_{i,P}\subseteq G_{P,0}$ is the $i$th ramification group in the upper numbering. For each $i$, choose $f_{i}\in H^{0}(\X_{i - 1},\orb_{\X_{i - 1}}(u_{i})|_{P_{i - 1}})$ corresponding to $F_{i}$, viewed as a germ of a rational function about $P_{i - 1}$, the preimage of $P$ in $\X_{i}$ (explicitly, one can restrict $F_{i}|_{\pi(U_{P})}$ and pull back to $\X_{i}$ to get $f_{i}$). By Theorem~\ref{thm:garjump}, each $f_{i}$ has valuation $u_{i}$ at $P$, where $u_{1},\ldots,u_{n}$ are the $n$ upper jumps in the ramification filtration $G_{P,0}\supseteq G_{P,1}\supseteq\cdots$. The isomorphism $\varphi$ follows as in Example~\ref{ex:key}; see also Remark~\ref{rem:ASWlocalpts}. 
\end{ex}

In general, every Artin--Schreier--Witt root stack $\Psi^{-1}((L,s,f_{1},\ldots,f_{n})/X)$ can be covered in the \'{e}tale topology by ``elementary'' ASW root stacks of the form $[Y/(\Z/p^{n}\Z)]$ as above. Rigorously: 

\begin{prop}
Let $\X = \Psi^{-1}((L,s,f_{1},\ldots,f_{n})/X)$ be an Artin--Schreier--Witt root stack of a scheme $X$ along a tuple $(L,s,f_{1},\ldots,f_{n})\in\DIV^{[1,m_{1},\ldots,m_{n}]}(X)$ and let $\pi : \X\rightarrow X$ be the coarse map. Then for any point $\bar{x} : \Spec k\rightarrow\X$, there is an \'{e}tale neighborhood $U$ of $x = \pi(\bar{x})$ such that $U\times_{X}\X\cong [Y/(\Z/p^{n}\Z)]$ where $Y$ is a smooth, projective Artin--Schreier--Witt cover of $U$. 
\end{prop}

\begin{proof}
Apply Lemma~\ref{lem:functorial} and Example~\ref{ex:elemASW}. See also \cite[Prop.~6.14]{kob}. 
\end{proof}

We are now ready to extend the classification results in \cite[Sec.~6]{kob} to wild stacky curves with $\Z/p^{n}\Z$ automorphism groups. We say a sequence of positive integers $m_{1},\ldots,m_{n}$ is {\it admissible} if it satisfies the conditions in \cite[Lem.~3.5]{op}, i.e.~if it is possible for $m_{1},\ldots,m_{n}$ to occur as the ramification jumps in the upper ramification filtration for a $\Z/p^{n}\Z$-extension of local fields. 

\begin{thm}
Let $\X$ be a Deligne--Mumford stack over a perfect field $k$ of characteristic $p > 0$ and let $m_{1},\ldots,m_{n}\geq 1$ be an admissible sequence. Then for any tuple $(\L,s,f_{1},\ldots,f_{n})\in\DIV^{[1,m_{1},\ldots,m_{n}]}(\X)$, the Artin--Schreier--Witt root stack $\Y = \Psi^{-1}((\L,s,f_{1},\ldots,f_{n})/\X)$ is also Deligne--Mumford. 
\end{thm}

\begin{proof}
Following the proof of \cite[Thm.~6.15]{kob}, it suffices to show this \'{e}tale-locally, say over an \'{e}tale neighborhood $U\rightarrow\X$. We may assume $\L$ is trivial over $U$ and lift $U\rightarrow [\overline{\W}_{n}(1,m_{1},\ldots,m_{n})/\W_{n}]$ to a map $U\rightarrow\overline{\W}_{n}(1,m_{1},\ldots,m_{n})$. Then Lemma~\ref{lem:functorial} and Example~\ref{ex:elemASW} imply $\Y\times_{\X}U \cong [Y/G]$ where $Y$ is a smooth scheme with an action of $G = \Z/p^{n}\Z$ making $Y$ into a $G$-torsor over $U$. Since $G$ is \'{e}tale, this quotient stack is Deligne--Mumford \cite[Cor.~8.4.2]{ols}, so $\Y\times_{\X}U$ is also Deligne--Mumford. 
\end{proof}

\begin{thm}
\label{thm:factorthroughASW}
Let $k$ be an algebraically closed field of characteristic $p > 0$ and suppose $\pi : Y\rightarrow X$ is a finite separable Galois cover of curves over $k$ with a ramification point $y\in Y$ over $x\in X$ such that the inertia group $I(y\mid x)$ is $\Z/p^{n}\Z$. Then there exist \'{e}tale neighborhoods $V\rightarrow Y$ of $y$ and $U\rightarrow X$ of $x$, a sequence of integers $m_{1},\ldots,m_{n}\geq 1$ satisfying the hypotheses of \cite{op}, and a tuple $(L,s,f_{1},\ldots,f_{n})\in\DIV^{[1,m_{1},\ldots,m_{n}]}(U)$ such that $V\rightarrow U$ factors through an Artin--Schreier--Witt root stack 
$$
V\rightarrow \Psi^{-1}((L,s,f_{1},\ldots,f_{n})/U)\rightarrow U. 
$$
\end{thm}

\begin{proof}
Both proofs of the $n = 1$ case from \cite{kob} generalize, but here's a streamlined version. Since $I = I(y\mid x) = \Z/p^{n}\Z$ is abelian, \cite[Prop.~VI.11.9]{ser1} prescribes a rational map $\varphi : X\dashrightarrow J_{\frak{m}}$ to a generalized Jacobian of $X$ with modulus $\frak{m}$ whose support includes $x$, such that $Y\cong X\times_{J_{\frak{m}}}J'$ for some cyclic, degree $p^{n}$ isogeny $J'\rightarrow J_{\frak{m}}$. Choose an \'{e}tale neighborhood $U'$ of $X$ on which $\varphi$ is defined and set $U = U'\cup\{x\}$. Then $\pi$, which is the pullback of $J'\rightarrow J_{\frak{m}}$, restricts to a one-point cover $\pi|_{V} : V\rightarrow U$ of degree $p^{n}$, ramified exactly at $x$, with Galois group $I$. We would like to extend this to a compactified Witt stack $\mathcal{W} := \overline{\W}_{n}(1,m_{1},\ldots,m_{n})$ for an admissible sequence $m_{1},\ldots,m_{n}$: 
\begin{center}
\begin{tikzpicture}[scale=2]
  \node at (-1,1) (a1) {$V$};
  \node at (0,1) (a) {$J'$};
  \node at (1,1) (b) {$\W_{n}$};
  \node at (2,1) (c) {$\mathcal{W}$};
  \node at (-1,0) (d1) {$U$};
  \node at (0,0) (d) {$J_{\frak{m}}$};
  \node at (1,0) (e) {$\W_{n}$};
  \node at (2,0) (f) {$\mathcal{W}$};
  \draw[->] (a1) -- (d1) node[left,pos=.5] {$\pi|_{V}$};
  \draw[->] (a1) -- (a);
  \draw[->] (a) -- (b);
  \draw[right hook ->] (b) -- (c);
  \draw[->] (d1) -- (d) node[above,pos=.5] {$\varphi$};
  \draw[->] (a) -- (d);
  \draw[->] (b) -- (e) node[right,pos=.5] {$\wp$};
  \draw[->] (c) -- (f) node[right,pos=.5] {$\Psi$};
  \draw[->] (d) -- (e);
  \draw[right hook ->] (e) -- (f);
\end{tikzpicture}
\end{center}

We may assume $\pi|_{V}$ is cut out by an Artin--Schreier--Witt equation $\wp\underline{y} = \underline{w}$ with $\underline{w}\in\W_{n}(k(U))$. For $1\leq i\leq n$, $m_{i} := v_{x}(w_{i})$ is the $i$th the upper jump in the ramification filtration of $I$. Let $(L,s)$ correspond to the divisor $x$ on $U$ and choose sections $f_{i}$ as in Example~\ref{ex:elemASW}. The data $(L,s,f_{1},\ldots,f_{n})$ defines the composition $U\rightarrow\mathcal{W}$ in the bottom row of the diagram. Pulling this data back to $V$ defines the composition in the upper row. Finally, by the definition of $\Psi^{-1}((L,s,f_{1},\ldots,f_{n})/U)$ as a pullback, we get a morphism $V\rightarrow \Psi^{-1}((L,s,f_{1},\ldots,f_{n})/U)$ through which $\pi|_{V}$ factors. 
\end{proof}

\begin{thm}
\label{thm:wildstacky}
Let $\X$ be a stacky curve over a perfect field $k$ of characteristic $p > 0$ with coarse space $X$ and let $x\in|\X|$ be a stacky point with automorphism group $\Z/p^{n}\Z$. Then $\X$ has an open substack containing $x$ of the form $\Psi^{-1}((L,s,f_{1},\ldots,f_{n})/U)$ where $U$ is an open subscheme of $X$ and $(L,s,f_{1},\ldots,f_{n})\in\DIV^{[1,m_{1},\ldots,m_{n}]}(U)$. 
\end{thm}

\begin{proof}
The ramification jumps of $\X$ at $x$ may be defined by pulling back to any \'{e}tale presentation $Y\rightarrow\X$ and reading off the upper jumps in the cover of curves $Y\rightarrow X$. We may take $U\subseteq X$ whose intersection with the image of the stacky locus of $\X$ is $\{x\}$. Set $\U = U\times_{X}\X$ and $V = U\times_{X}Y$. Then $V\rightarrow U$ is a one-point cover ramified at $x$, with inertia $\Z/p^{n}\Z$, so by Theorem~\ref{thm:factorthroughASW} the cover factors as $V\rightarrow \Psi^{-1}((L,s,f_{1},\ldots,f_{n})/U)\rightarrow U$, where $L = \orb_{U}(x)$ with distinguished section $s$, and $f_{1},\ldots,f_{n}$ come from an Artin--Schreier--Witt equation for the cover, as in Example~\ref{ex:elemASW}. By this description, we also get a map $\U\rightarrow\Psi^{-1}((L,s,f_{1},\ldots,f_{n})/U)$ which is independent of the cover chosen, so it gives us the desired substack. 
\end{proof}

\subsection{Cyclic-by-Tame Stacky Curves}
\label{sec:cyclicbytame}

For a stacky curve over a perfect field $k$ of characteristic $p > 0$, Proposition~\ref{prop:ramfilt}(a) implies that the automorphism group of a stacky point $x\in|\X|$ is of the form $P\rtimes\mu_{r}$, where $P$ is a finite \'{e}tale $p$-group scheme and $r$ is prime to $p$. When $P$ is cyclic and $r = 1$, Theorem~\ref{thm:wildstacky} characterizes the local geometry of $\X$ about $x$ in terms of geometric data, namely an Artin--Schreier--Witt root stack. In a similar fashion, Cadman's tame root stacks \cite{cad} characterize the local geometry about $x$ when $P$ is trivial and $r > 1$. These two parallel constructions can be combined as follows. 

\begin{thm}\label{thm:cyclicbytame}
Let $\X$ be a stacky curve over a perfect field $k$ of characteristic $p > 0$ with coarse space $X$. For any stacky point $x\in|\X|$ with cyclic-by-tame automorphism group $G \cong \Z/p^{n}\Z\rtimes\mu_{r}$, there is an open substack $\U\subseteq\X$ containing $x$ which is isomorphic to a fibre product of a wild and a tame root stack over an open subscheme of $X$. 
\end{thm}

\begin{proof}
Fix an \'{e}tale presentation $Y\rightarrow\X$. As in the proof of Theorem~\ref{thm:wildstacky}, we may take $\U = U\times_{X}\X$ where $U$ is an open subscheme of the coarse space only containing the image of a single stacky point, $x$. Setting $V = U\times_{X}Y$, we get a one-point cover of curves $V\rightarrow U$ with inertia group $G$ at any preimage of $x$. Using the same iterative argument as in Example~\ref{ex:key} then realizes $\U$ as a cyclic-by-tame root stack over $U$, as summarized in the following diagram: 
\begin{center}
\begin{tikzpicture}[xscale=2.8,yscale=1.2]
  \node at (0,2) (Y) {$V$};
  \node at (0,1) (YH) {$[V/G_{1}]$};
  \node at (0,0) (YG) {$[V/G]$};
  \node at (1,1) (X) {$V_{0}$};
  \node at (1,0) (XGH) {$[V_{0}/(G/G_{1})]$};
  \node at (2,0) (P1) {$U$};
  \draw[->] (Y) -- (YH);
  \draw[->] (Y) -- (X);
  \draw[->] (YH) -- (YG);
  \draw[->] (YH) -- (X) node[above,pos=.4] {$\delta$};
  \draw[->] (X) -- (XGH);
  \draw[->] (X) -- (P1);
  \draw[->] (YG) -- (XGH) node[above,pos=.5] {$\beta$};
  \draw[->] (YG) to[out=-40,in=220] (P1);
    \node at (.9,-.8) {$\gamma = \alpha\circ\beta$};
  \draw[->] (XGH) -- (P1) node[above,pos=.4] {$\alpha$};
\end{tikzpicture}
\end{center}

Explicitly, let $V_{0} = V/G_{1}$ be the intermediate curve fixed by the wild inertia group $G_{1}\cong\Z/p^{n}\Z$. Then $[V_{0}/(G/G_{1})] \cong [V_{0}/\mu_{r}]$ is a tame $r$th root stack over $U$, e.g.~by \cite[Ex.~2.4.1]{cad}; this constructs the map $\alpha$ in the diagram above. On the other hand, the map $\delta$ is a special case of Example~\ref{ex:elemASW} with a single stacky point above $x$, so $\delta$ realizes $[V/G_{1}]$ as an Artin--Schreier--Witt root stack over $V_{0}$. Finally, by Lemma~\ref{lem:doublequotient} we have $[V/G]\cong [[V/G_{1}]/(G/G_{1})]$, so $\beta$ constructs $[V/G]$ as an Artin--Schreier--Witt root stack over $[V_{0}/\mu_{r}]$, and together these give the cyclic-by-tame structure depicted by the map $\gamma$: 
\begin{align*}
    [Y/G] &\cong [[Y/G_{1}]/(G/G_{1})] \cong [\overline{\W}_{n}(\underline{m})/\W_{n}]\times_{[\overline{\W}_{n}(\underline{m})/\W_{n}]} [V_{0}/\mu_{r}]\\
        &\cong [\overline{\W}_{n}(\underline{m})/\W_{n}]\times_{[\overline{\W}_{n}(\underline{m})/\W_{n}]} [\A^{1}/\G_{m}]\times_{[\A^{1}/\G_{m}]} U. 
\end{align*}
This completes the proof. 
\end{proof}

More generally, if $\X$ has a stacky point $x$ with automorphism group $P\rtimes\mu_{r}$ where $P$ is an elementary abelian $p$-group, the same argument shows that $\X$ is \'{e}tale-locally a fibre product of Artin--Schreier(--Witt) root stacks about $x$. This is of evident interest in characteristic $2$ in light of the automorphism group of the point at $j = 0$ on the modular curve $\X(1)$ (see Example~\ref{ex:M11char2}). 

\begin{ex}
If $\X$ is a stacky curve in characteristic $p$ with a stacky point $x$ whose automorphism group is $G\cong \Z/p\Z\times\Z/p\Z$, one can obtain this local structure by iterating two Artin--Schreier root stacks. For example, a stacky $\P^{1}$ with a single stacky point at $\infty$ with this structure can be constructed by 
$$
\X = \wp_{m_{1}}^{-1}((\orb(1),x_{0},x_{1}^{m_{1}})/\P^{1})\times_{\P^{1}}\wp_{m_{2}}^{-1}((\orb(1),x_{0},x_{1}^{m_{2}})/\P^{1})
$$
where $m_{1},m_{2}$ are the \emph{lower} jumps in the desired ramification filtration of $G$. 
\end{ex}

It is not completely clear to the author how to further classify abelian-by-tame structures geometrically, since the semidirect product structure does not appear in the root stack constructions. In any event, the techniques in this article allow one to work ``from the ground up'' to construct any such root stack structure, in the style of \cite{gs}.


\section{A Universal Stack}
\label{sec:universalstack}

The various Artin--Schreier(--Witt) root stacks of a given scheme $X$ can be packaged together into a single stack as follows. We first deal with the Artin--Schreier case. 

Note that when $m\mid m'$, there is a morphism of weighted projective stacks $\P(1,m')\rightarrow\P(1,m)$ which is $\G_{a}$-equivariant, hence descending to $[\P(1,m')/\G_{a}]\rightarrow [\P(1,m)/\G_{a}]$. Denote the inverse limit of this system by $\mathcal{AS}$, which is an ind-algebraic stack. For a scheme $X$, the fibre product $\mathcal{AS}_{X} := \mathcal{AS}\times X$ parametrizes Artin--Schreier covers $Y\rightarrow X$. 

\begin{thm}
Let $Y\rightarrow X$ be a finite separable Galois cover of curves over an algebraically closed field of characteristic $p > 0$. Then about any ramification point with inertia group $\Z/p\Z$, the cover factors through $\mathcal{AS}_{U}$ for some \'{e}tale neighborhood $U$ of the corresponding branch point on $X$. 
\end{thm}

\begin{proof}
Apply \cite[Thm.~6.16]{kob}. 
\end{proof}

\begin{ex}
\label{ex:TY}
When $X = \Spec k((t))$ for a perfect field $k$ of characteristic $p > 0$, the stack $\mathcal{AS}_{X}$ coincides with the stack $\Delta_{\Z/p\Z}$ of formal $\Z/p\Z$-torsors studied in \cite{ty}. The quotients $[\P(1,m)/\G_{a}]$ can be viewed as a filtration of $\Delta_{\Z/p\Z}$ by ramification jump, coinciding with $(\A^{(S)})^{\infty}$ in the isomorphism $(\A^{(S)})^{\infty}\times B(\Z/p\Z)\cong\Delta_{\Z/p\Z}$ from [{\it loc.~cit.}, Thm.~4.13]. 
\end{ex}

More generally, for a fixed $n\geq 2$, the compactified Witt stacks $\overline{\W}_{n}(1,m_{1},\ldots,m_{n})$ form an inverse system via $m_{i}\mid m_{i}'$ for all $i$. Denote their inverse limit by $\mathcal{ASW}_{n}$, which is again an ind-algebraic stack. Let $\mathcal{ASW}_{n,X} := \mathcal{ASW}_{n}\times X$ be the stack which parametrizes ASW-covers of $X$. 

\begin{thm}
\label{thm:univASW}
Let $Y\rightarrow X$ be a finite separable Galois cover of curves over an algebraically closed field of characteristic $p > 0$. Then about any ramification point with inertia group $\Z/p^{n}\Z$, the cover factors through $\mathcal{ASW}_{n,U}$ for some \'{e}tale neighborhood $U$ of the corresponding branch point on $X$. 
\end{thm}

\begin{proof}
Apply Theorem~\ref{thm:factorthroughASW}. 
\end{proof}

\begin{ex}
As in Example~\ref{ex:TY}, $\mathcal{ASW}_{n,\Spec k((t))} \cong \Delta_{\Z/p^{n}\Z}$, the stack of formal $\Z/p^{n}\Z$-torsors also studied in \cite{ty}. In this case, the authors in {\it loc.~cit.}~do not give an explicit parametrization as in the $\Z/p\Z$ case, but they do present $\Delta_{\Z/p^{n}\Z}$ by a system of affine schemes. 
\end{ex}


\section{Application: Canonical Rings}
\label{sec:canring}

Recall from Theorem~\ref{thm:RH} that for a stacky curve $\X$ over a field $k$ with coarse moduli space $\pi : \X\rightarrow X$, the following formula defines a canonical divisor $K_{\X}$ on $\X$: 
$$
K_{\X} = \pi^{*}K_{X} + \sum_{x\in\X(k)}\sum_{i = 0}^{\infty} (|G_{x,i}| - 1)x
$$
where $G_{x,i}$ are the higher ramification groups in the lower numbering at $x$. 

\begin{ex}
\label{ex:ASWcanonicaldiv}
Let $Y\rightarrow\P^{1}$ be the Artin--Schreier--Witt cover given by the equations 
$$
y^{p} - y = \frac{1}{x^{m}} \quad\text{and}\quad z^{p} - z = \frac{y}{x^{m}}. 
$$
This cover is ramified at the point $Q$ lying over $\infty$ with group $G = \Z/p^{2}\Z$ and ramification jumps $m$ and $m(p^{2} + 1)$ by Example~\ref{ex:key}, so by the stacky Riemann--Hurwitz formula, the quotient stack $\X = [Y/G]$ has canonical divisor 
\begin{align*}
  K_{\X} &= -2Q + \sum_{i = 0}^{m} (p^{2} - 1)Q + \sum_{i = m + 1}^{m(p^{2} + 1)} (p - 1)Q\\
    &= -2Q + ((m + 1)(p^{2} - 1) + mp^{2}(p - 1))Q\\
    &= (mp^{3} + p^{2} - m - 3)Q. 
\end{align*}
Using the formula $\deg(K_{\X}) = 2g(\X) - 2$, we can also compute the genus of $\X$: 
$$
g(\X) = \frac{mp^{3} + p^{2} - m - 1}{2p^{2}}. 
$$
\end{ex}

Using an appropriate form of Riemann--Roch (see \cite[Cor.~1.189]{beh} or \cite[Rmk.~5.5.12]{vzb} or \cite[Sec.~7]{kob} for further discussion), one can recover the dimensions of the graded pieces of the canonical ring of $\X$: 
$$
h^{0}(\X,nK_{\X}) = \deg\left (\left\lfloor nK_{\X}\right\rfloor\right ) - g(X) + 1 + h^{0}(\X,(1 - n)K_{\X}). 
$$
See \cite[Ex.~7.8]{kob} for an example when $\X$ is an Artin--Schreier root stack over $\P^{1}$. This can then be used to construct a presentation of the canonical ring of $\X$, e.g.~by the main theorems in~\cite{od}. 

\begin{ex}
Let $\X = [Y/(\Z/p^{2}\Z)]$ be the Artin--Schreier--Witt quotient from Example~\ref{ex:ASWcanonicaldiv}. For the cases when $m < p^{2}$, we have 
$$
\lfloor K_{\X}\rfloor = -2H + \left\lfloor\frac{mp^{3} + p^{2} - m - 1}{p^{2}}\right\rfloor\infty = -2H + mp\infty
$$
so by Riemann--Roch, $h^{0}(\X,K_{\X}) = mp$. There's not such a clean formula for the global sections of $nK_{\X}$, but one still has
$$
h^{0}(\X,nK_{\X}) = -2n + \left\lfloor\frac{n(mp^{3} + p^{2} - m - 1)}{p^{2}}\right\rfloor + 1 = n(mp - 1) + \left\lfloor\frac{-n(m + 1)}{p^{2}}\right\rfloor. 
$$
When $m\geq p^{2}$, the formulas are even more complicated, reflecting the importance of the ramification jumps in the geometry of these wild stacky curves. 
\end{ex}

\begin{ex}
Let $\M_{1,1}$ be the moduli stack of elliptic curves over a field $F$ and let $\overline{\M}_{1,1}$ be its standard compactification obtained by adding nodal curves. When $\char F\not = 2,3$, $\overline{\M}_{1,1}$ is isomorphic to a stacky $\P^{1}$, namely the weighted projective stack $\P(4,6)$. While this is not a stacky curve, one can {\it rigidify} $\overline{\M}_{1,1}$ to remove the generic $\mu_{2}$ action and obtain a stacky curve $\overline{\M}_{1,1}^{\rig} \cong \P(2,3)$ (see \cite[Rmk.~5.6.8]{vzb}). This only changes the canonical ring by shifting the grading: a section in the weight $k$ piece of $R(\overline{\M}_{1,1}^{\rig})$ corresponds to a section in the weight $2k$ piece of $R(\overline{\M}_{1,1})$. The same is true if we instead consider the {\it log canonical ring} $R(\overline{\M}_{1,1},\Delta)$, where $\Delta$ is the log divisor of cusps (in this case, $\Delta$ is the single point added to compactify $\M_{1,1}$). By \cite[Lem.~6.2.3]{vzb}, 
$$
R(\overline{\M}_{1,1},\Delta) \cong \bigoplus_{k = 0}^{\infty} M_{k}
$$
where $M_{k}$ is the space of weight $k$ (Katz) modular forms. On the other hand, the isomorphism $\overline{\M}_{1,1}^{\rig}\cong\P(2,3)$ and Theorem~\ref{thm:RH} imply that $K = -2\infty + 2P + Q$ is a canonical divisor on $\overline{\M}_{1,1}^{\rig}$, where $P$ is the elliptic curve with $j = 0$ and $Q$ is the one with $j = 1728$. Then Riemann--Roch says that 
$$
R(\overline{\M}_{1,1}^{\rig},\Delta) \cong F[x_{2},x_{3}]
$$
where $x_{i}$ is a generator in weight $i$. Applying the grading shift, we get 
$$
R(\overline{\M}_{1,1},\Delta) \cong F[x_{4},x_{6}]
$$
which recovers a classical result for modular forms in all characteristics other than $2$ and $3$. 
\end{ex}

\begin{ex}
\label{ex:M11char3}
In characteristic $3$, the points on $\overline{\M}_{1,1}$ corresponding to elliptic curves with $j$-invariants $0$ and $1728$ collide, resulting in a more exotic stacky structure. Indeed, one can show that $\overline{\M}_{1,1}^{\rig}$ is isomorphic to a stacky curve with coarse space $\P^{1}$ and a single stacky point with automorphism group $S_{3}$, which is nonabelian. Such a stacky curve is of course not a tame or wild root stack, but by Theorem~\ref{thm:cyclicbytame}, one can take the fibre product of a tame square root stack and an Artin--Schreier root stack of order $3$, both over $\infty\in\P^{1}$, to obtain this curve. 
\end{ex}

\begin{ex}
\label{ex:M11char2}
In characteristic $2$, things are even worse. Once again, the points with $j = 0$ and $1728$ collide and this time $\overline{\M}_{1,1}^{\rig}$ is isomorphic to a stacky $\P^{1}$ with a single stacky point whose automorphism group is the semidirect product $(\Z/2\Z\times\Z/2\Z)\rtimes\Z/3\Z$. As the $2$-part of this group is not cyclic, one must iterate Artin--Schreier root stacks to achieve the wild part of the structure. 
\end{ex}

To use the stacky Riemann--Hurwitz formula (Theorem~\ref{thm:RH}) in both of these cases, one needs to compute the ramification filtration for the automorphism group at $j = 0 = 1728$ and read off the ramification jumps. In forthcoming joint work with David Zureick-Brown, we compute these ramification jumps and recover the result, originally due to Deligne \cite{del}, that in characteristics $p = 2,3$, the ring of mod $p$ modular forms (of level $1$) is isomorphic to the graded ring $\F_{p}[x_{1},x_{6}]$, where $x_{i}$ is a generator in degree $i$. We extend this computation to the family $\X_{0}(N)$ in characteristic $p = 2,3$, with $p\nmid N$, obtaining an algorithm to compute their log canonical rings. 

We will also give an account of the following example. 

\begin{ex}
Another example coming from modular curves is, for a prime $p > 5$, the quotient $\X = [X(p)/PSL_{2}(\F_{p})]$. As pointed out in \cite[Rmk.~5.3.11]{vzb}, in characteristic $3$, $\X$ is a stacky $\P^{1}$ with two stacky points $P$ and $Q$ whose automorphism groups are $\Z/p\Z$ and $S_{3}$, respectively (assuming $p > 3$). Therefore a canonical divisor on $\X$ is 
$$
K_{\X} = -2H + (p - 1)P + (5 + 2m)Q
$$
where $H\not\in\{P,Q\}$ and $m$ is the jump in the ramification filtration of $S_{3}$ at $Q$. Calculations show that $m = 1$ and the canonical ring of $\X$ is generated by monomials of the form $s^{a}t^{b}$, where $a$ and $b$ satisfy $\frac{(p + 1)b}{p}\leq a\leq \frac{7b}{6}$; see \cite{od}. In particular, the canonical ring has $\left\lfloor\frac{p}{6}\right\rfloor$ generators in degree $p$. For example, when $p = 7$ or $11$, the canonical ring has $1$ generator in degree $p$ and none in lower degrees. 
\end{ex}


\section{Future Directions}
\label{sec:future}

It would be desirable to have a geometric description (i.e.~in terms of intrinsic data such as line bundles and sections) of the local structure of stacky curves with arbitrary automorphism groups. As pointed out in Section~\ref{sec:introstackscharp}, these are all of the form $P\rtimes\mu_{r}$ for some \'{e}tale $p$-group scheme $P$ and some $r$ prime to $p$. Of course, Lemma~\ref{lem:functorial} and its tame analogue \cite[Rmk.~2.2.3]{cad} allow one to iterate tame and wild cyclic root stacks to obtain any local desired structure. In theory this can be used to describe such a structure in terms of line bundles and sections, but it is unwieldy. 

\begin{question}
Can one extend Theorem~\ref{thm:cyclicbytame} to arbitrary automorphism groups? 
\end{question}

\begin{question}
For example, can one give an intrinsic description (in terms of line bundles, sections, etc.) of a stacky $\P^{1}$ in characteristic $2$ with an automorphism group $Q_{8}$? 
\end{question}

From our perspective, the main obstacle to an intrinsic description of general stacky structures is the lack of a nonabelian generalization of Garuti's compactification $\overline{\W}_{n}$. A possible approach may be found in the Inaba classification of $G$-extensions, where $G$ is a $p$-group in characteristic $p$, due to Bell \cite{bel} in its most general form. 

\begin{thm}[{\cite[Thm.~1.5]{bel}}]
Let $G$ be a finite $p$-group, possibly nonabelian, and fix an embedding $G\hookrightarrow U_{n}(\F_{p})$ into the unitary group $U_{n}(\F_{p})$. For a ring $R$ of characteristic $p$ with connected spectrum $X = \Spec R$, the Galois $G$-covers of $X$ are classified up to isomorphism by the quotient $U_{n}(R)/LU_{n}(R)$, where $L(M) = M^{(p)}M^{-1}$ for any matrix $M\in U_{n}(R)$, and where $M^{(p)}$ is the matrix whose entries are the $p$th powers of the entries of $M$. 
\end{thm}

\begin{question}
Is there a natural compactification of the unitary group $U_{n}$, which contains Garuti's $\overline{\W}_{n}$ as a subvariety, such that the map $L : U_{n}\rightarrow U_{n}$ extends to the compactification? Is there a stacky compactification of $U_{n}$ generalizing the stacks $\overline{\W}_{n}(1,m_{1},\ldots,m_{n})$? 
\end{question}



\end{document}